\documentclass[a4paper,11pt]{amsart}
\usepackage[latin1]{inputenc}
\usepackage{amsmath}
\usepackage{amsfonts}
\usepackage{amssymb}
\usepackage{amsthm,graphicx}
\usepackage{mathrsfs}
\usepackage{xcolor}
\usepackage{fullpage,url}
\usepackage{soul,appendix}

\newcommand{\Q}{{\mathbb{Q}}}

\newcommand{\C}{\mathbb{C}}

\newcommand{\F}{\mathcal{F}}
\newcommand{\sumD}{ \sum_{K \in \F^\pm(X)}}
\newcommand{\sumDu}{ \sum_{K \in \F^\pm(u)}}

\newcommand{\D}{\mathfrak{D}}
\newcommand{\Ga}{\Gamma}
\newcommand{\vp}{\mathbf{p}}

\newcommand{\ve}{\mathbf{e}}
\newcommand{\vs}{\mathbf{s}}
\newcommand{\vk}{\mathbf{k}}

\newcommand{\eps}{\varepsilon}
\newcommand{\p}{\mathfrak p}

\def\Re{\operatorname{Re}}
\def\Im{\operatorname{Im}}

\newtheorem{X}{X}[section]
\newtheorem{corollary}[X]{Corollary}

\newtheorem{lemma}[X]{Lemma}

\newtheorem{proposition}[X]{Proposition}
\newtheorem{theorem}[X]{Theorem}
\newtheorem{conjecture}[X]{Conjecture}

\theoremstyle{definition}
\newtheorem{remark}[X]{Remark}

\numberwithin{equation}{section}

\usepackage[latin1]{inputenc}
\usepackage[T1]{fontenc}

\begin{document}

\title[Low-lying zeros of cubic Dedekind zeta functions]{Omega results for cubic field counts via lower-order terms in the one-level density}

\author{Peter J. Cho, Daniel Fiorilli, Yoonbok Lee and Anders S\"odergren}

\date{\today}

\address{Department of Mathematical Sciences, Ulsan National Institute of Science and Technology, Ulsan, Korea}
\email{petercho@unist.ac.kr}
\address{CNRS, Universit\'e Paris-Saclay, Laboratoire de math\'ematiques d'Orsay, 91405, Orsay, France}
\email{daniel.fiorilli@universite-paris-saclay.fr}
\address{Department of Mathematics, Research Institute of Basic Sciences, Incheon National University, Korea}
\email{leeyb@inu.ac.kr, leeyb131@gmail.com}
\address{Department of Mathematical Sciences, Chalmers University of Technology and the University of Gothenburg, SE-412 96 Gothenburg, Sweden}
\email{andesod@chalmers.se} 
\thanks{Peter J. Cho is supported by  the NRF grant funded by the  Korea government(MSIT) (No. 2019R1F1A1062599) and the Basic  Science Research Program(2020R1A4A1016649). Yoonbok Lee is supported by  the NRF grant funded by the  Korea government(MSIT)  (No.2019R1F1A1050795). Daniel Fiorilli was supported at the University of Ottawa by an NSERC discovery grant. Anders S\"odergren was supported by a grant from the Swedish Research Council (grant 2016-03759).}

\maketitle

\begin{abstract}
In this paper we obtain a precise formula for the $1$-level density of $L$-functions attached to non-Galois cubic Dedekind zeta functions. We find a secondary term
which is unique to this context, in the sense that no lower-order term of this shape has appeared in previously studied families. The presence of this new term allows us to deduce an omega result for cubic field counting functions, under the assumption of the Generalized Riemann Hypothesis. We also investigate the associated $L$-functions Ratios Conjecture, and find that it does not predict this new lower-order term.  Taking into account the secondary term in Roberts' Conjecture, we refine the Ratios Conjecture to one which captures this new term. Finally, we show that any improvement in the exponent of the error term of the recent Bhargava--Taniguchi--Thorne cubic field counting estimate would imply that the best possible error term in the refined Ratios Conjecture is $O_\eps(X^{-\frac 13+\eps})$. This is in opposition with all previously studied families, in which the expected error in the Ratios Conjecture prediction for the $1$-level density is $O_\eps(X^{-\frac 12+\eps})$. 
\end{abstract}

\section{Introduction}

In \cite{KS1,KS2} Katz and Sarnak made a series of fundamental conjectures about statistics of low-lying zeros in families of $L$-functions. Recently, these conjectures have been refined by Sarnak, Shin and Templier~\cite{SaST} for families of parametric $L$-functions. There is a huge body of work on the confirmation of these conjectures for particular test functions in various families, many of which are harmonic (see, e.g., \cite{ILS,Ru,FI,HR,ST}). There are significantly fewer geometric families that have been studied. In this context we mention the work of Miller~\cite{M} and Young~\cite{Y2} on families of elliptic curve $L$-functions, and that of Yang~\cite{Y}, Cho and Kim~\cite{CK1,CK2} and Shankar, S\"odergren and Templier~\cite{SST1} on families of Artin $L$-functions. 

In families of Artin $L$-functions, these results are strongly linked with counts of number fields. More precisely, the set of admissible test functions is determined by the quality of the error terms in such counting functions. In this paper we consider the sets 
$$\F^\pm(X):=\{K/\Q \text{ non-Galois}\,:\, [K:\Q]=3, 0<\pm D_K < X\},$$
where for each cubic field $K/\Q$ of discriminant $D_K$ we include only one of its three isomorphic copies. The first power-saving estimate for the cardinality $N^\pm(X):=|\F^\pm (X)|$
was obtained by Belabas, Bhargava and Pomerance~\cite{BBP}, and was later refined by Bhargava, Shankar and Tsimerman~\cite{BST}, Taniguchi and Thorne~\cite{TT}, and Bhargava, Taniguchi and Thorne~\cite{BTT}. The last three of these estimates take the shape
\begin{equation} 
N^\pm(X) = C_1^\pm X + C_2^\pm  X^{\frac 56} + O_{\eps}(X^{\theta+\eps})
\label{equation TT}
\end{equation}
for certain explicit values of $ \theta <\frac 56$, implying in particular Roberts' conjecture \cite{Ro}. Here,
$$ C_1^+:=\frac{1}{12\zeta(3)};\hspace{.5cm} C_2^+:=\frac{4\zeta(\frac 13)}{5\Gamma( \frac 23)^3\zeta(\frac 53)}; \hspace{.5cm} C_1^-:=\frac{1}{4\zeta(3)};\hspace{.5cm} C_2^-:=\frac{4\sqrt 3\zeta(\frac 13)}{5\Gamma( \frac 23)^3\zeta(\frac 53)}.$$
The presence of this secondary term is a striking feature of this family, and we are interested in studying its consequences for the distribution of low-lying zeros. More precisely, the estimate~\eqref{equation TT} suggests that one should be able to extract a corresponding lower-order term in various statistics on those zeros.

In addition to~\eqref{equation TT}, we will consider precise estimates involving local conditions, which are of the form
\begin{align}\label{equation TT local 2}
  N^\pm_{p} (X,T) : &=  \#\{ K \in \mathcal F^\pm(X) : p \text{ has splitting type }T \text{ in } K\} \notag
\\&  = A^\pm_p (T ) X +B^\pm_p(T) X^{\frac 56} + O_{\eps}(p^{ \omega}X^{\theta+\eps}),
\end{align}  
where $p$ is a given prime, $T$ is a splitting type, and the constants $A^\pm_p (T )$ and $B^\pm_p (T )$ are defined in Section~\ref{section background}. Here, $\theta$ is the same constant as that in~\eqref{equation TT}, and $\omega\geq 0$. Note in particular that~\eqref{equation TT local 2} implies~\eqref{equation TT} (take $p=2$ in~\eqref{equation TT local 2} and sum over all splitting types $T$).

Perhaps surprisingly, it turns out that the study of low-lying zeros has an application to cubic field counts. More precisely, we were able to obtain the following conditional omega result for $N^\pm_{p} (X,T)$.
\begin{theorem}
\label{theorem omega result counts}
Assume the Generalized Riemann Hypothesis for $\zeta_K(s)$ for each $K\in\mathcal F^{\pm}(X)$. If $\theta,\omega\geq 0$ are admissible values in~\eqref{equation TT local 2}, then $\theta+\omega\geq \frac 12$.
\end{theorem}

As part of this project, we have produced numerical data which suggests that $\theta=\frac 12$ and any $\omega >0$ are admissible values in~\eqref{equation TT local 2} (indicating in particular that the bound $\omega+\theta \geq \frac 12$ in Theorem~\ref{theorem omega result counts} could be best possible). We have made several graphs to support this conjecture in Appendix~\ref{appendix}. As a first example of these results, in Figure~\ref{figure intro} we display a graph of $X^{-\frac 12}(N^+_{5} (X,T)-A^+_5 (T ) X -B^+_5(T) X^{\frac 56} ) $ for the various splitting types $T$, which suggests that $\theta=\frac 12$ is admissible and best possible. 

\begin{figure}[h]
\label{figure intro}

\begin{center}
\includegraphics[scale=.55]{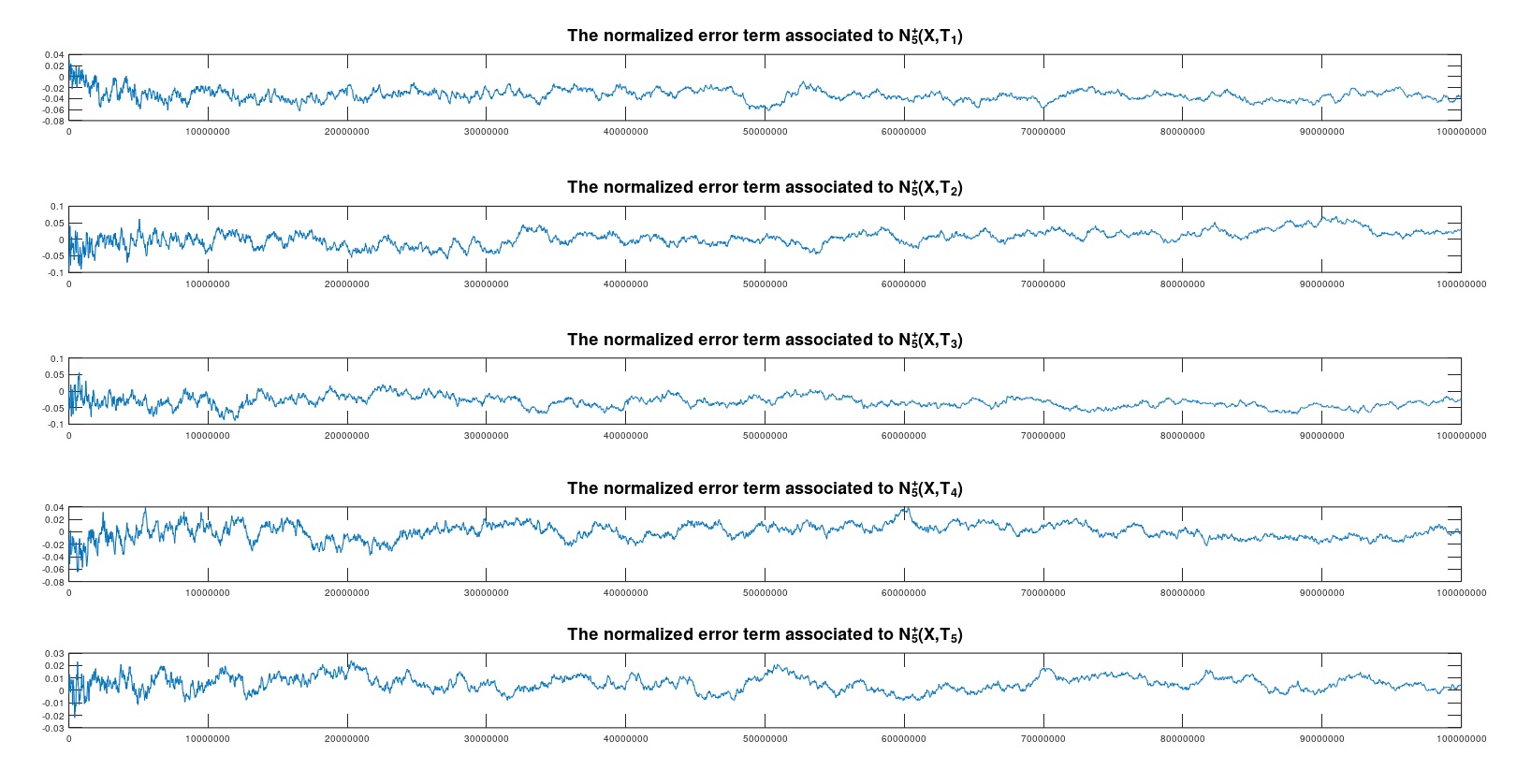} 

\end{center}\caption{The normalized error terms $X^{-\frac 12}(N^+_{5} (X,T)-A^+_5 (T ) X -B^+_5(T) X^{\frac 56} ) $ for the splitting types $T= T_1,\dots,T_5$ as described in Section~\ref{section background}. }
\end{figure}

Let us now describe our unconditional result on low-lying zeros. For a cubic field $K$, we will focus on the Dedekind zeta function $\zeta_K(s)$, whose
 $1$-level density is defined by
$$ \D_{\phi}(K) :=  \sum_{\gamma_{K}} \phi\left(  \frac{\log (X/(2 \pi e)^2)}{2\pi} \gamma_K\right).$$
Here, $\phi$ is an even, smooth and rapidly decaying real function for which the Fourier transform
$$\widehat \phi(\xi) :=\int_{\mathbb R} \phi(t) e^{-2\pi i \xi t} dt $$ is compactly supported. Note that $\phi$ can be extended to an entire function through the inverse Fourier transform. Moreover, $X$ is a parameter (approximately equal to $|D_K|$) and $\rho_K=\frac 12+i\gamma_K$ runs through the non-trivial zeros\footnote{The Riemann Hypothesis for $\zeta_K(s)$ implies that $\gamma_K\in \mathbb R$.} of $\zeta_K(s)/\zeta(s)$. In order to understand the distribution of the $\gamma_K$, we will average $\D_{\phi}(K)$ over the family $\mathcal F^\pm(X)$. Our main technical result is a precise estimation of this average. 

\begin{theorem}
\label{theorem main} 
Assume that the cubic field count~\eqref{equation TT local 2} holds for some fixed parameters $\frac 12\leq \theta <\frac 56$ and $ \omega\geq0$.
Then, for any real even Schwartz function $\phi$ for which $\sigma:=\sup({\rm supp}(\widehat \phi))< \frac {1-\theta}{\omega+ \frac 12}$, we have the estimate
\begin{multline}\frac 1{N^\pm(X)} \sumD \D_{\phi}(K)=\widehat \phi(0)\Big(1 + \frac{  \log (4 \pi^2 e)  }{L} -\frac{C_2^\pm}{5C_1^\pm} \frac{X^{-\frac 16}}{L} 
+ \frac{(C_2^\pm)^2 }{5(C_1^\pm)^2 } \frac{X^{-\frac 13}}{L} \Big) \\
+  \frac1{\pi}\int_{-\infty}^{\infty}\phi\Big(\frac{Lr}{2\pi}\Big)\Re\Big(\frac{\Ga'_{\pm}}{\Ga_{\pm}}(\tfrac12+ir)\Big)dr   -\frac{2}{L}\sum_{p,e}\frac{x_p\log p}{p^{\frac e2}}\widehat\phi\Big(\frac{\log p^e}{L}\Big) (\theta_e+\tfrac 1p)\\  -\frac{2C_2^\pm X^{-\frac 16}}{C_1^\pm L}\Big( 1-\frac{C_2^\pm}{C_1^\pm}X^{-\frac{1}{6}}\Big) \sum_{p,e}\frac{\log p}{p^{\frac e2}}\widehat\phi\Big(\frac{\log p^e}{L}\Big) \beta_e(p) +O_{\eps}(X^{\theta - 1 + \sigma(\omega + \frac12 ) +\eps}) , 
\label{equation main theorem}
\end{multline}
where $\Ga_+(s):=
\pi^{-s}\Gamma(\frac s2)^2$, $\Ga_-(s):=
\pi^{-s}\Gamma(\frac s2)\Gamma(\frac {s+1}2)$, $x_p:=(1+\frac 1p+\frac 1{p^2})^{-1}$, $\theta_e$ and $\beta_e(p)$ are defined in~\eqref{definition theta_e} and~\eqref{eqaution definition beta}, respectively, and $L:=\log \big( \frac{X}{(2 \pi e)^2}\big)$.
\end{theorem}

\begin{remark}
In the language of the Katz--Sarnak heuristics, the first and third terms on the right-hand side of~\eqref{equation main theorem} are a manifestation of the symplectic symmetry type of the family $\F^\pm(X)$. More precisely, one can turn~\eqref{equation main theorem} into an expansion in descending powers of $L$ using Lemma~\ref{one-level-2-term} as well as~\cite[Lemma 12.14]{MV}. The first result in this direction is due to Yang~\cite{Y}, who showed that under the condition $\sigma<\frac 1{50}$, we have that
\begin{equation}
 \frac 1{N^\pm(X)}\sumD \D_{\phi}(K) 
 =\widehat \phi(0)-\frac{\phi(0)}2+o_{X\rightarrow \infty}(1).
\label{Yang}
\end{equation}
This last condition was relaxed to $\sigma<\frac 4{41}$ by Cho--Kim~\cite{CK1,CK2}\footnote{In \cite{CK1}, the condition $\sigma < \frac{4}{25}$ should be corrected to $\sigma<\frac 4{41}$.} and Shankar--S\"odergren--Templier~\cite{SST1}, independently, and corresponds to the admissible values $\theta=\frac 79$ and $\omega=\frac {16}9$ in~\eqref{equation TT} and~\eqref{equation TT local 2} (see~\cite{TT}). In the recent paper~\cite{BTT}, Bhargava, Taniguchi and Thorne show that $\theta=\frac 23$ and $\omega=\frac 23$ are admissible, and deduce that~\eqref{Yang} holds as soon as $\sigma<\frac 2{7}$. Theorem~\ref{theorem main} refines these results by obtaining a power saving estimate containing lower-order terms for the left-hand side of~\eqref{Yang}. Note in particular that the fourth term on the right-hand side of~\eqref{equation main theorem} is of order $   X^{\frac{\sigma-1}6+o(1)}$ (see once more Lemma~\ref{one-level-2-term}). 
\end{remark}

The Katz--Sarnak heuristics are strongly linked with statistics of eigenvalues of random matrices, and have been successful in predicting the main term in many families. However, this connection does not encompass lower-order terms. The major tool for making predictions in this direction is the $L$-functions Ratios Conjecture of Conrey, Farmer and Zirnbauer~\cite{CFZ}. In particular, these predictions are believed to hold down to an error term of size roughly the inverse of the square root of the size of the family. As an example, consider the unitary family of Dirichlet $L$-functions modulo $q$, in which the Ratios Conjecture's prediction is particularly simple. It is shown in~\cite{G+} that if $ \eta$ is a real even Schwartz function for which $\widehat \eta$ has compact (but arbitrarily large) support, then this conjecture implies the estimate
\begin{multline}
   \frac 1{\phi(q)}\sum_{\chi \bmod q}\sum_{\gamma_{\chi}} \eta\Big(  \frac{\log q}{2\pi} \gamma_\chi\Big) = \widehat{\eta}(0) \Big(  1-\frac { \log(8\pi e^{\gamma})}{\log q}-\frac{\sum_{p\mid q}\frac{\log p}{p-1}}{\log q}\Big)  +\int_0^{\infty}\frac{\widehat{\eta}(0)-\widehat{\eta}(t)}{q^{\frac t2}-q^{-\frac t2}} dt + E(q),
   \end{multline}
   where $\rho_\chi=\frac 12+i\gamma_\chi$ is running through the non-trivial zeros of $L(s,\chi)$, and $E(q) \ll_{\eps} q^{-\frac 1 2+\eps}$. In~\cite{FM}, it was shown that this bound on $E(q)$ is essentially best possible in general, but can be improved when the support of $\widehat \eta$ is small. This last condition also results in improved error terms in various other families (see, for instance, \cite{MiSymplectic,MiOrthogonal,FPS1,FPS2,DFS}). 
     
Following the Ratios Conjecture recipe, we can obtain a prediction for the average of $\D_\phi(K)$ over the family $\F^\pm(X)$. The resulting conjecture, however, differs from Theorem~\ref{theorem main} by a term of order $X^{\frac{\sigma-1}6+o(1)} $, which is considerably larger than the expected error term $O_{\eps}(X^{-\frac 12 +\eps})$. We were able to isolate a specific step in the argument which could be improved in order to include this additional contribution. More precisely, modifying Step 4 in~\cite[Section 5.1]{CFZ},
we recover a refined Ratios Conjecture which predicts a term of order $X^{\frac{\sigma-1}6+o(1)}$, in agreement with Theorem~\ref{theorem main}.

\begin{theorem}
\label{theorem RC}
Let $\frac 12\leq \theta<\frac 56$ and $\omega\geq0 $ be such that~\eqref{equation TT local 2} holds. Assume Conjecture~\ref{ratios-thm} on the average of shifts of the logarithmic derivative of $\zeta_K(s)/\zeta(s)$, as well as the Riemann Hypothesis for $\zeta_K(s)$, for all $K\in \F^{\pm}(X)$.  Let $\phi$ be a real even Schwartz function such that $ \widehat \phi$ is compactly supported. Then we have the estimate
\begin{multline*}
\frac 1{N^\pm(X)}\sumD \sum_{\gamma_K}\phi \Big(\frac{L\gamma_K}{2\pi }\Big)=\widehat \phi(0)\Big(1 + \frac{  \log (4 \pi^2 e)  }{L} -\frac{C_2^\pm}{5C_1^\pm} \frac{X^{-\frac 16}}{L} 
+ \frac{(C_2^\pm)^2 }{5(C_1^\pm)^2 } \frac{X^{-\frac 13}}{L} \Big) \\+  \frac1{\pi}\int_{-\infty}^{\infty}\phi\Big(\frac{Lr}{2\pi}\Big)\Re\Big(\frac{\Ga'_{\pm}}{\Ga_{\pm}}(\tfrac12+ir)\Big)dr   -\frac{2}{L}\sum_{p,e}\frac{x_p\log p}{p^{\frac e2}}\widehat\phi\Big(\frac{\log p^e}{L}\Big) (\theta_e+\tfrac 1p) \\-\frac{2C_2^\pm X^{-\frac 16}}{C_1^\pm L}\Big( 1-\frac{C_2^\pm}{C_1^\pm}X^{-\frac{1}{6}}\Big) \sum_{p,e}\frac{\log p}{p^{\frac e2}}\widehat\phi\Big(\frac{\log p^e}{L}\Big) \beta_e(p)+J^\pm(X)
+ O_\eps(X^{\theta-1+\eps}),
\end{multline*}
where $J^\pm(X)$ is defined in \eqref{J X def}.
If $\sigma=\sup ( {\rm supp}(\widehat\phi)) <1$, then we have the estimate 
\begin{equation}\label{J X asymp form}
J^\pm(X) =  C^{\pm}   X^{-\frac 13} \int_{\mathbb R} \Big( \frac{X}{(2\pi e)^2 }\Big)^{\frac{\xi}6} \widehat \phi(\xi) d\xi +O_\eps( X^{\frac{\sigma-1}2+\eps}),
\end{equation}
 where $C^{\pm} $ is a nonzero absolute constant which is defined in~\eqref{equation definition C}. Otherwise, we have the identity
\begin{multline}
 J^\pm(X)=-\frac{1}{\pi i} \int_{(\frac 15)} \phi \Big(\frac{Ls}{2\pi i}\Big) \Big( 1-\frac{C_2^\pm}{C_1^\pm}X^{-\frac{1}{6}}\Big)   X^{-s}  \frac{\Ga_{\pm}(\frac 12-s)}{\Ga_{\pm}(\frac 12+s)}\zeta(1-2s) \frac{ A_3( -s, s) }{1-s} ds \\
 -\frac{1}{\pi i} \int_{(\frac 1{20})} \phi \Big(\frac{Ls}{2\pi i}\Big)\frac{C_2^\pm }{C_1^\pm} X^{-s-\frac 16} \frac{\Ga_{\pm}(\frac 12-s)}{\Ga_{\pm}(\frac 12+s)}\zeta(1-2s)  \Big\{ \Big( 1-\frac{C_2^\pm}{C_1^\pm}X^{-\frac{1}{6}}\Big)    \frac{\zeta(\tfrac 56-s)}{\zeta(\tfrac 56+s)}\frac{ A_4(-s,s)}{ 1- \frac{ 6s}5}
 \\+\frac{C_2^\pm}{C_1^\pm}X^{-\frac 16}  \frac{A_3( -s, s) }{1-s}  
 \Big\}ds,
 \label{equation definition J(X)}
\end{multline} 
where $A_3(-s,s)$ and $A_4(-s,s)$ are defined in~\eqref{A3 prod form} and~\eqref{equation definition A_4}, respectively.
\end{theorem}

\begin{remark}
It is interesting to compare Theorem~\ref{theorem RC} with Theorem~\ref{theorem main}, especially when $\sigma$ is small. Indeed, for $\sigma<1$, the difference between those two evaluations of the $1$-level density is given by 
$$  C^\pm  X^{-\frac 13}\int_{\mathbb R} \Big( \frac{X}{(2\pi e)^2 }\Big)^{\frac{\xi}6} \widehat \phi(\xi) d\xi+O_\eps\big(X^{\frac{\sigma-1}{2}+\eps}+X^{\theta - 1 + \sigma(\omega + \frac12 ) +\eps}\big).$$
Selecting test functions $\phi$ for which $\widehat \phi \geq 0$ and $\sigma$ is positive but arbitrarily small, this shows that no matter how large $\omega$ is, any admissible $\theta<\frac 23$ in~\eqref{equation TT} and~\eqref{equation TT local 2} would imply that this difference is asymptotic to $C^\pm X^{-\frac 13}\int_{\mathbb R} ( \frac{X}{(2\pi e)^2 })^{\frac{\xi}6} \widehat \phi(\xi) d\xi \gg X^{-\frac 13}$. In fact, Roberts' numerics~\cite{Ro} (see also~\cite{B}), as well as our numerical investigations described in Appendix~\ref{appendix}, indicate that $\theta= \frac 12$  could be admissible in~\eqref{equation TT}
and~\eqref{equation TT local 2}. In other words, in this family, the Ratios Conjecture, as well as our refinement (combined with the assumption of~\eqref{equation TT} and~\eqref{equation TT local 2} for some $\theta<\frac 23$ and $\omega\geq0$), are not sufficient to obtain a prediction with precision $o(X^{-\frac 13})$. This is surprising, since Conrey, Farmer and Zirnbauer have conjectured this error term to be of size $O_\eps(X^{-\frac 12+\eps})$, and this has been confirmed in several important families~\cite{MiSymplectic,MiOrthogonal,FPS1,FPS2,DFS} (for a restricted set of test functions). 
\end{remark}

\section*{Acknowledgments}
We would like to thank Frank Thorne for inspiring conversations, and for sharing the preprint  \cite{BTT} with us. We also thank Keunyoung Jeong for providing us with preliminary computational results.
 The computations in this paper were carried out on a personal computer using pari/gp as well as Belabas'  CUBIC program. We thank Bill Allombert for his help regarding the latest development version of pari/gp.

\section{Background}
\label{section background}

Let $K/\Q$ be a non-Galois cubic field, and let $\widehat{K}$ be the Galois closure of $K$ over $\mathbb{Q}$. Then, the Dedekind zeta function of the field $K$ has the decomposition 
\begin{align*}
\zeta_K(s)  = \zeta(s) L(s,\rho,\widehat{K}/\mathbb{Q}),
\end{align*}
where $L(s,\rho,\widehat{K}/\mathbb{Q})$ is the Artin $L$-function associated to the two-dimensional representation $\rho$ of ${\rm Gal}(\widehat K/\Q) \simeq S_3$. The strong Artin conjecture is known for such representations; in this particular case we have an explicit underlying cuspidal representation $\tau$ of $GL_2/\mathbb{Q}$ such that $L(s,\rho,\widehat{K}/\mathbb{Q})=L(s,\tau)$.  For the sake of completeness, let us describe $\tau$ in more detail. Let $F=\mathbb{Q}[\sqrt{D_K}]$, and let $\chi$ be a non-trivial character of Gal$(\widehat{K}/F)\simeq C_3$, considered as a Hecke character of $F$. Then $\tau=\text{Ind}^{\mathbb Q}_{F} \chi$ is a dihedral representation of central character $\chi_{D_K}=\big(\frac {D_K}\cdot \big)$. When $D_K <0$, $\tau$ corresponds to a weight one newform of level $\left| D_K \right|$ and nebentypus $\chi_{D_K}$, and when $D_K>0$ it corresponds to a weight zero Maass form (see \cite[Introduction]{DFI}). In both cases, we will denote the corresponding automorphic form by $f_K$, and in particular we have the equality
$$L(s,\rho,\widehat{K}/\mathbb{Q}) = L(s,f_K).$$

We are interested in the analytic properties of $\zeta_K(s)/\zeta(s)=L(s,f_K)$. 
We have the functional equation 
\begin{equation}
\label{equation functional equation}
\Lambda(s,f_K)=\Lambda(1-s,f_K).
\end{equation}Here, 
$\Lambda(s,f_K):= |D_K|^{\frac s2}  \Ga_{f_K}(s) L(s,f_K)$ is the completed $L$-function, with the gamma factor
$$ \Ga_{f_K}(s)=\begin{cases} \Ga_+(s) & \text{ if } D_K>0 \text{ } (\text{that is } K \text{ has signature } (3,0) ); \\ \Ga_-(s) & \text{ if } D_K <  0 \text{ } (\text{that is } K \text{ has signature } (1,1) ) ,
\end{cases}$$
where $\Ga_+(s):=
\pi^{-s}\Gamma(\frac s2)^2$ and $\Ga_-(s):=
\pi^{-s}\Gamma(\frac s2)\Gamma(\frac {s+1}2)$.

The coefficients of $L(s,f_K)$ have an explicit description in terms of the splitting type of the prime ideal $(p)\mathcal O_{K}$. 
Writing 
$$ L(s,f_K) = \sum_{n =1}^\infty \frac{\lambda_K (n)}{n^s},$$
we have that
\begin{center}
\begin{tabular}{c|c|c}
Splitting type & $(p)\mathcal O_{K} $ & $\lambda_K(p^e)$\\
\hline
$T_1$ & $\p_1\p_2\p_3 $ & $ e+1$  \\
$T_2$ & $\p_1\p_2 $ & $(1+(-1)^e)/2$ \\
$T_3$ & $\p_1$ & $\tau_e$ \\
$T_4$ & $\p_1^2\p_2 $ & $1$ \\
$T_5$ & $\p_1^3$ & $0$, \\
\end{tabular}
\end{center}
where 
$$  \tau_e:= \begin{cases}
1 &\text{ if } e\equiv 0 \bmod 3; \\
-1 &\text{ if } e\equiv 1 \bmod 3; \\
0 &\text{ if } e\equiv 2 \bmod 3.
\end{cases} $$
Furthermore, we find that the coefficients of the reciprocal
\begin{equation}\label{lemobius}
 \frac{1}{L(s,f_K)} = \sum_{n=1}^\infty \frac{ \mu_K (n)}{n^s}
\end{equation}
are given by
\begin{equation*} 
\mu_{K}(p^k) =
\begin{cases}
-\lambda_K(p) & \mbox{~if~} k = 1;\\
\Big( \frac{D_K} p\Big) & \mbox{~if~} k = 2;\\
0 & \mbox{~if~} k > 2.
\end{cases}
\end{equation*}
The remaining values of $\lambda_K(n)$ and $\mu_K(n)$ are determined by multiplicativity.
Finally, the coefficients of the logarithmic derivative
 $$  -\frac{L'}{L}(s,f_K)=\sum_{n\geq 1} \frac{\Lambda(n)a_K(n)}{n^s}$$
 are given by
\begin{center}
\begin{tabular}{c|c|c}
Splitting type & $(p) $ & $a_K(p^e)$ \\
\hline
$T_1$ & $\p_1\p_2\p_3 $ & $ 2$ \\
$T_2$ & $\p_1\p_2 $ & $1+(-1)^e$ \\
$T_3$ & $\p_1$ & $\eta_e$ \\
$T_4$ & $\p_1^2\p_2 $ & $1$ \\
$T_5$ & $\p_1^3$ & $0$, \\
\end{tabular}
\end{center}
where 
$$ \eta_e:= \begin{cases}
2 &\text{ if } e \equiv 0 \bmod 3; \\
-1 &\text{ if }  e \equiv \pm 1 \bmod 3.
\end{cases} $$

We now describe explicitly the constants $A_p^\pm(T)$ and $B_p^\pm(T)$ that appear in~\eqref{equation TT local 2}. More generally, let $\vp = ( p_1 , \ldots , p_J)$ be a vector of primes, and let $\vk = ( k_1  ,\ldots , k_J )\in \{1,2,3,4,5\}^J $. (When $J=1$, $\vp=(p)$ is a scalar and we will abbreviate by writing $\vp=p$, and similarly for $\vk$.) We expect that
\begin{align}\label{equation TT local 2 vector form}
  N^\pm_{\vp} (X,T_\vk) : &=  \#\{ K \in \mathcal F^\pm(X) :  p_j \text{ has splitting type }T_{k_j} \text{ in }K \; (1\leq j\leq J)\} \notag
\\&  = A^\pm_\vp (T_\vk ) X +B^\pm_\vp(T_\vk) X^{\frac 56} + O_{\eps}((p_1 \cdots p_J )^{ \omega}X^{\theta+\eps}),
\end{align}  
for some $\omega \geq0$ and with the same $\theta$ as in \eqref{equation TT}.
Here,
$$ A^\pm_\vp ( T_\vk ) = C_1^\pm \prod_{j = 1}^J ( x_{p_j} c_{k_j} (p_j ) ),\hspace{1cm} B^\pm_\vp ( T_\vk ) = C_2^\pm \prod_{j = 1}^J ( y_{p_j} d_{k_j} (p_j ) ) ,$$
$$x_p:=\Big(1+\frac 1p+\frac 1{p^2}\Big)^{-1}, \hspace{2cm} y_p:=\frac{1-p^{-\frac 13}}{(1-p^{-\frac 53})(1+p^{-1})}, $$
and  $c_k(p)$ and $d_k(p)$ are defined in the following table:
\begin{center}
\begin{tabular}{c|c|c}
$k$ & $c_k(p)$ & $d_k(p)$ \\
\hline
$1$ & $\frac 16 $ & $\frac{(1+p^{-\frac 13})^3}6$ \\
$2$ & $\frac 12 $ & $\frac{(1+p^{-\frac 13})(1+p^{-\frac 23})}{2} $ \\
$3$ & $\frac 13$ & $\frac{(1+p^{-1})}{3} $ \\
$4$ & $\frac 1{p} $ & $\frac{(1+p^{-\frac 13})^2}{p} $ \\
$5$ & $\frac 1{p^2}$ & $\frac{(1+p^{-\frac 13})}{p^2}  .$ 
\end{tabular}
\end{center}
Recently, Bhargava, Taniguchi and Thorne~\cite{BTT} have shown that the values $\theta=\omega=\frac{2}3$ are admissible in~\eqref{equation TT local 2 vector form}.

\section{New lower-order terms in the $1$-level density}

In this section we shall estimate the 1-level density
$$ \frac 1{N^\pm(X)}\sumD \D_{\phi}(K)   $$
assuming  the cubic field count~\eqref{equation TT local 2}  for some fixed parameters  $\frac 12\leq \theta <\frac 56$ and $ \omega\geq 0$. Throughout the paper we will use the shorthand 
$$ L =\log \Big(\frac X{(2\pi e)^2}\Big). $$ 
The starting point of this section is the explicit formula. 

\begin{lemma}
\label{lemma explicit formula}
Let $\phi$ be a real even Schwartz function whose Fourier transform is compactly supported, and let $K\in \F^\pm (X)$. We have the formula
\begin{equation}\label{explicit formula}\begin{split}
 \D_\phi(K) = \sum_{\gamma_K}\phi\Big(\frac{L\gamma_K}{2\pi}\Big)=& \frac{\widehat\phi(0)}{L } \log{|D_K|} +  \frac1{\pi}\int_{-\infty}^{\infty}\phi\Big(\frac{Lr}{2\pi}\Big)\Re\Big(\frac{\Ga'_{\pm}}{\Ga_{\pm}}(\tfrac12+ir)\Big)dr \\
& -\frac2{L }\sum_{n=1}^{\infty}\frac{\Lambda(n)}{\sqrt n}\widehat\phi\Big(\frac{\log n}{L}\Big)  a_K(n),
\end{split}\end{equation}
where $\rho_K=\frac 12+i\gamma_K$ runs over the non-trivial zeros of $L(s,f_K)$. 
\end{lemma}
\begin{proof}
This follows from e.g.~\cite[Proposition 2.1]{RS}, but for the sake of completeness we reproduce the proof here. By Cauchy's integral formula, we have the identity
\begin{multline*}
 \sum_{\gamma_K}\phi\Big(\frac{L\gamma_K}{2\pi}\Big)=\frac{1}{2\pi i} \int_{(\frac 32)}\phi\left( \frac{L}{2\pi i} \left( s - \frac{1}{2} \right)\right) \frac{\Lambda'}{\Lambda}(s,f_K)ds \\-\frac{1}{2\pi i} \int_{(-\frac 12)}\phi\left( \frac{L}{2\pi i} \left( s - \frac{1}{2} \right)\right) \frac{\Lambda'}{\Lambda}(s,f_K) ds.
\end{multline*}
These integrals converge since $\phi\left(\frac{L}{2 \pi i}\left( s-\frac{1}{2} \right) \right)$ is rapidly decreasing in vertical strips.
For the second integral, we apply the change of variables $s \rightarrow 1-s$. Then, by the functional equation in the form $ \frac{\Lambda'}{\Lambda}(1-s,f_K)=-\frac{\Lambda'}{\Lambda}(s,f_K)$ and since $\phi(-s)=\phi(s)$, we deduce that
\begin{align*}
 \sum_{\gamma_K}\phi\Big(\frac{L\gamma_K}{2\pi}\Big)=\frac{1}{\pi i} \int_{(\frac 32)} \phi\left( \frac{L}{2\pi i} \left( s - \frac{1}{2} \right)\right) \frac{\Lambda'}{\Lambda}(s,f_K)ds.  
\end{align*}
Next, we insert the identity $$ \frac{\Lambda'}{\Lambda}(s,f_K)=\frac{1}{2}\log |D_K| + \frac{\Gamma_{f_K}'}{\Gamma_{f_K}}(s) - \sum_{n\geq 1}\frac{\Lambda(n)a_K(n)}{n^s}$$ and separate into three integrals. By shifting the contour of integration to $\Re(s)=\frac12$ in the first two integrals, we obtain the first two terms on the right-hand side of~\eqref{explicit formula}. The third integral is equal to
\begin{align*}
-2\sum_{n\geq 1} \frac{\Lambda(n)a_K(n)}{\sqrt{n}}\frac{1}{2\pi i} \int_{(\frac 32)} \phi\left( \frac{L}{2\pi i} \left( s - \frac{1}{2} \right)\right)n^{-(s-\frac 12)}ds.
\end{align*}
By moving the contour to $\Re(s)=\frac 12$ and applying Fourier inversion, we find the third term on the right-hand side of~\eqref{explicit formula} and the claim follows.   
\end{proof}

Our goal is to average \eqref{explicit formula} over $K \in \F^\pm(X)$. We begin with the first term. 

\begin{lemma}
\label{lemma average log}
Assume that \eqref{equation TT} holds for some $0\leq\theta<\frac 56$. Then, we have the estimate
 $$\frac{1}{ N^\pm(X)}\sumD    \log{|D_K|}    =  \log X  -   1  -\frac{C_2^\pm}{5C_1^\pm}  X^{-\frac 16}   +  \frac{(C_2^\pm)^2}{ 5  (C_1^\pm)^2 }     X^{- \frac13}   +O_{\eps}(X^{\theta-1+\eps}+X^{-\frac 12}).$$
\end{lemma}
\begin{proof}
Applying partial summation, we find that
\begin{align*}
\sumD \log |D_K| = \int_1^X (\log t) dN^\pm(t) =  N^\pm(X)\log X-N^\pm(X)-\frac 15 C_2^\pm X^{\frac 56} +O_{\eps}(X^{\theta+\eps}). 
\end{align*}
The claimed estimate follows from applying~\eqref{equation TT}.
\end{proof}

For the second term of~\eqref{explicit formula}, we note that it is constant on $\F^\pm(X)$. We can now concentrate our efforts on the average of the third (and most crucial) term
\begin{equation}
\label{equation definitino I}
 I^{\pm}(X;\phi) := -\frac2{LN^\pm(X) }\sum_{p} \sum_{e=1}^\infty \frac{\log p }{p^{e/2}} \widehat\phi\Big(\frac{e \log p}{L}\Big)   \sumD  a_K(p^e ). 
\end{equation}
It follows from~\eqref{equation TT local 2} that
\begin{align} \notag
\sumD a_K(p^e)&=2N^\pm_p(X,T_1)+(1+(-1)^e)N^\pm_p(X,T_2)+ \eta_e N^\pm_p(X,T_3)+N^\pm_p(X,T_4)\\
&= C_1^\pm X (\theta_e+\tfrac 1p)x_p+ C_2^\pm X^{\frac 56} (1+p^{-\frac 13})(\kappa_e(p)+p^{-1}+p^{-\frac 43})y_p +O_{\eps}(p^{\omega }X^{\theta +\eps}), \label{sum p e 1}
\end{align}
where
\begin{equation}
\theta_e:=\delta_{2\mid e}+\delta_{3\mid e}=\begin{cases}
2 &\text{ if } e\equiv 0 \bmod 6 \\
0 &\text{ if } e\equiv 1 \bmod 6 \\
1 &\text{ if } e\equiv 2 \bmod 6 \\
1 &\text{ if } e\equiv 3 \bmod 6 \\
1 &\text{ if } e\equiv 4 \bmod 6 \\
0 &\text{ if } e\equiv 5 \bmod 6, 
\end{cases}
\label{definition theta_e}
\end{equation} 
and
\begin{equation}
\kappa_e(p):=(\delta_{2\mid e}+\delta_{3\mid e})(1+p^{-\frac 23}) + (1-\delta_{3\mid e})p^{-\frac 13}=\begin{cases}
2+2p^{-\frac 23} &\text{ if } e\equiv 0 \bmod 6 \\
p^{-\frac 13} &\text{ if } e\equiv 1 \bmod 6 \\
1+p^{-\frac 13}+p^{-\frac 23} &\text{ if } e\equiv 2 \bmod 6 \\
1+p^{-\frac 23} &\text{ if } e\equiv 3 \bmod 6 \\
1+p^{-\frac 13}+p^{-\frac 23} &\text{ if } e\equiv 4 \bmod 6 \\
p^{-\frac 13} &\text{ if } e\equiv 5 \bmod 6.
\end{cases}
\label{definition kappa_e}
\end{equation} 
Here, $ \delta_{\mathcal P}  $ is equal to $1$ if $ \mathcal P$ is true, and is equal to $0$ otherwise. Note that we have the symmetries $\theta_{-e} = \theta_e$ and $\kappa_{-e}(p) = \kappa_e(p)$. With this notation, we prove the following proposition. 

\begin{proposition}
\label{proposition prime sum}
Let $\phi$ be a real even Schwartz function for which $\widehat \phi$ has compact support and let $\sigma:=\sup({\rm supp}(\widehat \phi))$. Assume that~\eqref{equation TT local 2} holds for some fixed parameters $0\leq  \theta <\frac 56$ and $\omega\geq 0$. Then we have the estimate
\begin{multline*}
  I^{\pm}(X;\phi)= -\frac{2}{L}\sum_{p,e}\frac{x_p\log p}{p^{\frac e2}}\widehat\phi\Big(\frac{\log p^e}{L}\Big) (\theta_e+\tfrac 1p) \\
    + \frac{2}{L}\bigg( -  \frac{ C_2^\pm }{C_1^\pm  }X^{-\frac 16} +   \frac{ ( C_2^\pm )^2 }{(C_1^\pm )^2  }X^{-\frac 13} \bigg) \sum_{p,e}\frac{\log p}{p^{\frac e2}}\widehat\phi\Big(\frac{\log p^e}{L}\Big) \beta_e(p) +O_{\eps}(X^{\theta - 1 + \sigma(\omega + \frac12 ) +\eps}+X^{-\frac 12+\frac{\sigma}6}), 
  \end{multline*}
where 
\begin{equation}
\beta_e(p):=y_p (1+p^{-\frac 13})(\kappa_e(p) +p^{-1}+p^{-\frac 43})-x_p(\theta_e+\tfrac 1p).
\label{eqaution definition beta}
\end{equation}
\end{proposition}
\begin{proof}

Applying~\eqref{sum p e 1}, we see that
\begin{multline*}
 I^{\pm}(X;\phi) =  -\frac{2C_1^\pm X}{LN^\pm(X)}\sum_p \sum_{e=1}^\infty \frac{x_p\log p}{p^{\frac e2}}\widehat\phi\Big(\frac{\log p^e}{L}\Big) (\theta_e+\tfrac 1p)\\ 
 -  \frac{2C_2^\pm X^{\frac 56}}{LN^\pm(X)} \sum_p \sum_{e=1}^\infty \frac{y_p(1+p^{-\frac 13})\log p}{p^{\frac e2}}\widehat\phi\Big(\frac{\log p^e}{L}\Big) (\kappa_e(p) +p^{-1}+p^{-\frac 43}) \\
+O_{\eps}\Big(X^{\theta - 1 +\eps}\sum_{\substack{p^e \leq X^\sigma \\ e\geq 1}}p^{\omega -\frac e2}\log p \Big)\\
 =  -\frac{2 }{L }\bigg( 1 - \frac{ C_2^\pm}{C_1^\pm} X^{- \frac16} + \frac{ (C_2^\pm)^2 }{(C_1^\pm)^2 } X^{- \frac13} \bigg)  \sum_p \sum_{e=1}^\infty \frac{x_p\log p}{p^{\frac e2}}\widehat\phi\Big(\frac{\log p^e}{L}\Big) (\theta_e+\tfrac 1p)\\ 
 - \frac{2}{L}  \bigg(  \frac{C_2^\pm }{C_1^\pm}X^{-\frac 16}  - \frac{ (C_2^\pm)^2 }{(C_1^\pm)^2 } X^{- \frac13}  \bigg) \sum_p \sum_{e=1}^\infty \frac{y_p(1+p^{-\frac 13})\log p}{p^{\frac e2}}\widehat\phi\Big(\frac{\log p^e}{L}\Big) (\kappa_e(p) +p^{-1}+p^{-\frac 43}) \\ +O_{\eps}\big(X^{\theta - 1 + \sigma(\omega + \frac12 ) +\eps} +X^{-\frac 12+\frac{\sigma}6}\big).
\end{multline*}
 Note in particular that the error term $O(X^{-\frac 12+\frac{\sigma}6})$ bounds the size of the contribution of the first omitted term in the expansion of $X^{\frac56}/N^{\pm}(X)$ appearing in the second double sum above. Indeed, this follows since $\kappa_1(p)=p^{-\frac13}$ and 
\begin{equation*}
X^{-\frac12}\sum_{p\leq X^{\sigma}}\frac{\log p}{p^{\frac56}}=O\big(X^{-\frac 12+\frac{\sigma}6}\big).
\end{equation*} 
The claimed estimate follows. 
\end{proof}

\begin{proof}[Proof of Theorem~\ref{theorem main}]
Combine Lemmas~\ref{lemma explicit formula} and ~\ref{lemma average log} with Proposition~\ref{proposition prime sum}.
\end{proof}

We shall estimate $ I^{\pm}(X;\phi)$ further, and find asymptotic expansions for the double sums in Proposition~\ref{proposition prime sum}.
 
\begin{lemma} \label{one-level-2-term} 
Let $\phi$ be a real even Schwartz function whose Fourier transform is compactly supported, define $\sigma:=\sup({\rm supp}(\widehat \phi))$, and let $\ell$ be a positive integer. Define
$$ I_1(X;\phi) :=   \sum_{p}\sum_{e=1}^\infty  \frac{x_p\log p}{p^{\frac e2}}\widehat\phi\Big(\frac{\log p^e}{L}\Big) (\theta_e+\tfrac 1p)  , \qquad  I_2(X;\phi) :=   \sum_{p}\sum_{e=1}^\infty  \frac{\log p}{p^{\frac e2}}\widehat\phi\Big(\frac{\log p^e}{L}\Big) \beta_e(p) .$$
Then, we have the asymptotic expansion
$$
I_1(X;\phi)  =   \frac{\phi (0)}{4} L      + \sum_{n=0 }^\ell  \frac{ \widehat{\phi}^{(n)} (0)  \nu_1 (n) }{ n!}  \frac{ 1 }{L^n }      + O_{\ell} \Big(   \frac{1 }{L^{\ell+1} }  \Big), 
$$
where 
\begin{multline*} \nu_1 (n )   :=   \delta_{n=0} +  \sum_{p}\sum_{  e \neq 2 }   \frac{x_p e^n ( \log p)^{n+1} }{p^{\frac e2}} (\theta_e+\tfrac 1p) + \sum_p  \frac{ 2^n  ( \log p)^{n+1}  }{p} \Big( x_p\Big( 1 + \frac1p\Big) -1 \Big)   \\
  +   \int_1^\infty  \frac{2^n ( \log u)^{n-1} ( \log u - n )  }{u^2}     \mathcal{R} (u) du 
\end{multline*}
with $ \mathcal{R}(u) := \sum_{ p \leq u } \log p -  u $. Moreover, we have the estimate
$$I_2(X;\phi)  =    L \int_0^\infty \widehat{\phi} ( u)  e^{\frac{Lu}6} du +  O\Big( X^{\frac \sigma 6} e^{-c_0(\sigma)\sqrt{\log X}}\Big),$$
where $c_0(\sigma)>0$ is a constant. Under RH, we have the more precise expansion
$$I_2(X;\phi)  =    L \int_0^\infty \widehat{\phi} ( u)  e^{\frac{Lu}{6}} du +  \sum_{n=0}^\ell  \frac{ \widehat{\phi}^{(n)} (0)  \nu_2 (n) }{ n!}  \frac{ 1 }{L^n }      + O_{\ell} \Big(   \frac{1 }{L^{\ell+1} }  \Big), $$
where
\begin{multline*}  
\nu_2 (n )   :=    \delta_{n=0}  + \sum_{p}\sum_{e=2}^\infty  \frac{e^n (\log p )^{n+1}  \beta_e(p) }{p^{\frac e2}} +
\sum_{p}   \frac{(\log p)^{n+1}}{p^{\frac 12}}   \Big(  \beta_1 (p) -  \frac{1}{p^{\frac13}} \Big)\\
+ \int_1^{\infty}  \frac{( \log u)^{n-1} ( 5 \log u - 6n)  }{6 u^{\frac{11}{6} }}    \mathcal{R} (u) du. 
\end{multline*} 
\end{lemma}

\begin{proof}
We first split the sums as 
\begin{equation}\label{prime sum eqn 1}
  I_1(X;\phi)    =  \sum_p \frac{\log p }{p} \widehat{\phi} \Big(  \frac{ 2  \log p }{L} \Big)   + I'_1(X;\phi),  \qquad 
  I_2(X;\phi)       =\sum_p \frac{ \log p }{ p^{\frac56}}\widehat\phi\Big(\frac{\log p }{L}\Big)  +I'_2(X;\phi) ,
\end{equation}
where
\begin{equation}\begin{split}\label{prime sum eqn 3} 
I'_1(X;\phi) & :=   
\sum_{p}\sum_{e \neq 2 }   \frac{x_p\log p}{p^{\frac e2}}\widehat\phi\Big(\frac{\log p^e}{L}\Big) (\theta_e+\tfrac 1p) + \sum_p  \frac{   \log p }{p} \Big( x_p\Big( 1 + \frac1p\Big) -1 \Big) \widehat{\phi} \Big(  \frac{ 2  \log p }{L} \Big),  \\
I'_2(X;\phi) & :=  \sum_{p}\sum_{e=2}^\infty  \frac{\log p}{p^{\frac e2}}\widehat\phi\Big(\frac{\log p^e}{L}\Big) \beta_e(p)  +
\sum_{p}   \frac{\log p}{p^{\frac 12}}\widehat\phi\Big(\frac{\log p }{L}\Big) \Big(  \beta_1 (p) -  \frac{1}{p^{\frac13}} \Big) .
\end{split}\end{equation}
We may also rewrite the sums in~\eqref{prime sum eqn 1} using partial summation as follows:
\begin{equation}\begin{split}\label{prime sum eqn 2}
 \sum_p \frac{\log p }{p} \widehat{\phi} \Big(  \frac{ 2  \log p }{L} \Big)  & =   \int_1^\infty \frac{1 }{u} \widehat{\phi} \Big(  \frac{ 2  \log u }{L} \Big)d( u+ \mathcal{R}(u))  \\
&    =  \frac{ \phi(0) }{4} L +  \widehat{\phi}(0)  -  \int_1^\infty \Big(\frac{-1}{u^2} \widehat{\phi} \Big( \frac{2 \log u }{L} \Big) + \frac{2}{u^2L} \widehat{\phi}' \Big( \frac{2 \log u }{L} \Big)  \Big)  \mathcal{R} (u) du , \\
 \sum_p \frac{\log p }{p^{\frac56} } \widehat{\phi} \Big(  \frac{   \log p }{L} \Big) &  =     L \int_0^\infty \widehat{\phi} ( u)  e^{Lu/6} du +  \widehat{\phi}(0)   \\
 & \hspace{2cm}-  \int_1^{X^{\sigma}} \Big(\frac{-5}{6 u^2} \widehat{\phi} \Big( \frac{  \log u }{L} \Big) + \frac{ 1 }{u^2L} \widehat{\phi}' \Big( \frac{  \log u }{L} \Big)  \Big) u^{\frac16}  \mathcal{R} (u) du .
\end{split}\end{equation}

Next, for any $\ell\geq 1$ and $ |t| \leq \sigma  $, Taylor's theorem reads 
\begin{equation}\label{taylor eqn}
  \widehat{\phi} ( t ) = \sum_{n=0}^\ell  \frac{ \widehat{\phi}^{(n)} (0) }{ n!} t^n + O_{\ell} ( |t|^{\ell+1}),
\end{equation}
and one has a similar expansion for $\widehat \phi'$. The claimed estimates follow from substituting this expression into~\eqref{prime sum eqn 3} and \eqref{prime sum eqn 2} and evaluating the error term using the prime number theorem $ \mathcal{R}(u)\ll u e^{-c \sqrt{ \log u }} $. 
\end{proof}
 
We end this section by proving Theorem~\ref{theorem omega result counts}.

\begin{proof}[Proof of Theorem~\ref{theorem omega result counts}]
Assume that $\theta,\omega\geq 0$ are admissible values in~\eqref{equation TT local 2}, and are such that $\theta+\omega <\frac 12$. Let $\phi$ be any real even Schwartz function such that $\widehat \phi\geq 0$ and $1< \sup({\rm supp}(\widehat \phi)) < (\frac 56-\theta)/(\frac 13+\omega)$; this is possible thanks to the restriction $\theta+\omega <\frac 12$. Combining  Lemmas~\ref{lemma explicit formula} and ~\ref{lemma average log} with Proposition~\ref{proposition prime sum}, we obtain the estimate\footnote{This is similar to the proof of Theorem~\ref{theorem main}. However, since we have a different condition on $\theta$ (that is $\theta + \omega < \frac 12$), there is an additional error term in the current estimate.}
\begin{multline}\frac 1{N^\pm(X)} \sumD \D_{\phi}(K)=\widehat \phi(0)\Big(1 + \frac{  \log (4 \pi^2 e)  }{L} -\frac{C_2^\pm}{5C_1^\pm} \frac{X^{-\frac 16}}{L} 
+ \frac{(C_2^\pm)^2 }{5(C_1^\pm)^2 } \frac{X^{-\frac 13}}{L} \Big) \\
+  \frac1{\pi}\int_{-\infty}^{\infty}\phi\Big(\frac{Lr}{2\pi}\Big)\Re\Big(\frac{\Ga'_{\pm}}{\Ga_{\pm}}(\tfrac12+ir)\Big)dr   -\frac{2}{L}\sum_{p,e}\frac{x_p\log p}{p^{\frac e2}}\widehat\phi\Big(\frac{\log p^e}{L}\Big) (\theta_e+\tfrac 1p)\\  -\frac{2C_2^\pm X^{-\frac 16}}{C_1^\pm L}\Big( 1-\frac{C_2^\pm}{C_1^\pm}X^{-\frac{1}{6}}\Big) \sum_{p,e}\frac{\log p}{p^{\frac e2}}\widehat\phi\Big(\frac{\log p^e}{L}\Big) \beta_e(p) +O_{\eps}(X^{\theta - 1 + \sigma(\omega + \frac12 ) +\eps}+X^{-\frac 12+\frac \sigma 6}), 
\label{equation corollary}
\end{multline}
 where $\sigma=\sup({\rm supp}(\widehat \phi))$. 
 
To bound the integral involving the gamma function in~\eqref{equation corollary}, we note that Stirling's formula implies that for $s$ in any fixed vertical strip minus discs centered at the poles of $\Gamma_\pm(s)$, we have the estimate
$$ \Re \left(  \frac{ \Gamma'_{\pm} }{ \Gamma_{\pm}} ( s) \right) =  \log | s | + O(1).$$
Now, $ \phi(x) \ll |x|^{-2}$,  and thus
$$ \frac{1 }{ \pi} \int_{-\infty}^\infty \phi \Big(  \frac{ Lr}{ 2 \pi} \Big)  \Re \Big(  \frac{ \Gamma'_{\pm} }{ \Gamma_{\pm}} ( \tfrac12 + ir ) \Big) dr   \ll \int_{-1}^1 \left|  \phi \Big(  \frac{ Lr}{ 2 \pi} \Big) \right| dr + \int_{ |r|\geq 1 }   \frac{ \log (1+|r|)}{  (Lr)^2} dr \ll \frac1L. $$
Moreover, Lemma~\ref{one-level-2-term} implies the estimates 
$$- \frac{2}{L}\sum_{p,e}\frac{x_p\log p}{p^{\frac e2}}\widehat\phi\Big(\frac{\log p^e}{L}\Big) (\theta_e+\tfrac 1p) \ll 1$$
and
\begin{multline*}
-\frac{2C_2^\pm X^{-\frac 16}}{C_1^\pm L}\Big( 1-\frac{C_2^\pm}{C_1^\pm}X^{-\frac{1}{6}}\Big) \sum_{p,e}\frac{\log p}{p^{\frac e2}}\widehat\phi\Big(\frac{\log p^e}{L}\Big) \beta_e(p)  \\= -\frac{2C_2^\pm X^{-\frac 16}}{C_1^\pm } \int_0^\infty \widehat \phi(u) e^{\frac{Lu}{6}} du +O\big( X^{-\frac { 1}6}+X^{\frac{\sigma -2}{6}} \big), 
\end{multline*}
since the Riemann hypothesis for $\zeta_K(s)$ implies the Riemann hypothesis for $\zeta(s)$. Combining these estimates, we deduce that
the right-hand side of~\eqref{equation corollary} is
$$ \leq  -C_\eps X^{\frac{\sigma-1}6-\eps}+O_\eps(1+X^{\frac{\sigma-1}6-\delta+\eps}+X^{-\frac 13+\frac \sigma 6}), $$
where $\eps>0$ is arbitrary, $C_\eps$ is a positive constant, and $\delta :=\frac{\sigma-1}6 -( \theta - 1 + \sigma(\omega + \frac12 )) >0 $. However, for small enough $\eps$, this contradicts the bound 
$$ \frac 1{N^\pm(X)} \sumD \D_{\phi}(K)  = O( \log X), $$
which is a direct consequence of the Riemann Hypothesis for $\zeta_K(s)$ and the Riemann-von Mangoldt formula~\cite[Theorem 5.31]{IK}.
\end{proof}

\section{A refined Ratios Conjecture}
\label{section RC}

The celebrated $L$-functions Ratios Conjecture \cite{CFZ} predicts precise formulas for estimates of averages of ratios of (products of) $L$-functions evaluated at points close to the critical line. The conjecture is presented in the form of a recipe with instructions on how to produce predictions of a certain type in any family of $L$-functions. In order to follow the recipe it is of fundamental importance to have control of counting functions of the type~\eqref{equation TT} and \eqref{equation TT local 2 vector form} related to the family. The connections between counting functions, low-lying zeros and the Ratios Conjecture are central in the present investigation. 

The Ratios Conjecture has a large variety of applications. Applications to problems about low-lying zeros first appeared in the work of Conrey and Snaith \cite{CS}, where they study the one-level density of families of quadratic Dirichlet $L$-functions and quadratic twists of a holomorphic modular form. The investigation in \cite{CS} has inspired a large amount of work on low-lying zeros in different families; see, e.g., \cite{MiSymplectic,MiOrthogonal,HKS,FM,DHP,FPS1,FPS3,MS,CP, Wax}.

As part of this project, we went through the steps of the Ratios Conjecture recipe with the goal of estimating the $1$-level density. We noticed that the resulting estimate does not  predict certain terms in Theorem~\ref{theorem main}.
To fix this, we modified~\cite[Step 4]{CFZ}, which is the evaluation of the average of the coefficients appearing in the approximation of the expression 
\begin{equation} \label{ralphgam}
R(\alpha,\gamma;X):=\frac 1{N^\pm(X)}\sumD \frac{L\left(\frac{1}{2}+\alpha,f_K\right)}{L\left(\frac{1}{2}+\gamma,f_K\right)}.
\end{equation}
More precisely, instead of only considering the main term, we kept track of the secondary term in Lemma~\ref{lemma count1}. 
 
We now describe more precisely the steps in the Ratios Conjecture recipe. The first step involves the approximate functional equation for $L(s,f_K)$, which reads 
\begin{equation}
L\left(s, f_K\right) = \sum_{n<x} \frac{\lambda_K(n)}{n^s}+|D_K|^{\frac 12-s}\frac{\Ga_{\pm}(1-s)}{\Ga_{\pm}(s)} \sum_{n<y}
\frac{\lambda_K(n)}{n^{1-s}}+\:{\rm Error},\label{approxeq}
\end{equation}
where $x,y$ are such that $xy\asymp |D_K| (1+|t|)^2 $(this is in analogy with~\cite{CS}; see~\cite[Theorem 5.3]{IK} for a description of the approximate functional equation of a general $L$-function). The analysis will be carried out assuming that the error term can be neglected, and that the sums can be completed.

Following \cite{CFZ}, we replace the
numerator of $\eqref{ralphgam}$ with the approximate functional equation $\eqref{approxeq}$ and the
denominator of $\eqref{ralphgam}$ with $\eqref{lemobius}$. We will  need to estimate the first sum in \eqref{approxeq} evaluated at $s=\frac{1}{2}+\alpha$, where $|\Re(\alpha)|$ is sufficiently small. This gives the contribution
\begin{equation} \label{ragone}
R_1(\alpha, \gamma;X):=\frac1{N^\pm(X)}\sumD \sum_{h,m} \frac{\lambda_K(m)\mu_{K}(h)}{m^{\frac{1}{2} +
\alpha}h^{\frac{1}{2}+\gamma}}
\end{equation}
to \eqref{ralphgam}. This infinite sum converges absolutely in the region $\Re(\alpha)>\frac 12$ and $\Re(\gamma)>\frac 12$, however, later in this section we will provide an analytic continuation to a wider domain. We will also need to evaluate the contribution of the second sum in $\eqref{approxeq}$, which is given by
\begin{equation} \label{ragtwo}
R_2(\alpha, \gamma;X):=\frac1{N^\pm(X)}\sumD |D_K|^{-\alpha}\frac{\Ga_{\pm}\left(\frac 12-\alpha \right)}{\Ga_{\pm}\left(\frac 12+\alpha\right)}\sum_{h,m} \frac{\lambda_K(m)\mu_{K}(h)}{m^{\frac{1}{2}-\alpha}h^{\frac{1}{2}+\gamma}}.
\end{equation}
(Once more, the series converges absolutely for $\Re(\alpha)<-\frac 12$ and $\Re(\gamma)>\frac 12$, but we will later provide an analytic continuation to a wider domain.)

A first step in the understanding of the $R_j(\alpha, \gamma;X)$  will be achieved using the following precise evaluation of the expected value of $\lambda_K(m)\mu_K(h) $.

\begin{lemma}\label{lemma count1} 
Let $m,h\in \mathbb N$, and let $\frac 12 \leq \theta<\frac 56$ and $\omega\geq0$ be such that~\eqref{equation TT local 2 vector form} holds. Assume that $h$ is cubefree. We have the estimate
\begin{multline*} 
\frac{1}{N^\pm(X)}\sumD\lambda_K(m)\mu_K(h)  = \prod_{p^e\parallel m, p^s \parallel h}f(e,s,p)x_p\\
+ \Bigg(  \prod_{p^e\parallel m, p^s \parallel h}g(e,s,p)y_p-  \prod_{p^e\parallel m, p^s \parallel h}f(e,s,p)x_p\Bigg) \frac{C_2^\pm}{C_1^\pm}X^{-\frac{1}{6}}\Big( 1-\frac{C_2^\pm}{C_1^\pm}X^{-\frac{1}{6}}\Big) \\
+O_\eps\bigg( \prod_{  p \mid hm , p^e \parallel m  } \big( (2 e + 5)p^{\omega}\big) X^{\theta-1 +\eps}\bigg),
 \end{multline*} 
where  
\begin{align*}
f(e,0,p)&:=  \frac{e+1}{6}+\frac{1+(-1)^e}{4}+\frac{\tau_e}{3}+\frac{1}{p}; \\
 f(e,1,p)&:= - \frac{e+1}{3}+\frac{\tau_e}{3}- \frac{1}{p} ; \\
  f(e,2,p)&:=\frac{e+1}{6}-\frac{1+(-1)^e}{4}+\frac{\tau_e}{3} ; \\
g(e,0,p)&:= \frac{(e+1)(1+p^{-\frac 13})^3}{6} +\frac{(1+(-1)^e)(1+p^{-\frac 13})(1+p^{-\frac 23})}{4} \\
             & \quad + \frac{\tau_e(1+p^{-1})}{3} + \frac{(1+p^{-\frac 13})^2}{p}; \\
 g(e,1,p)&:= -\frac{(e+1)(1+p^{-\frac 13})^3}{3}+ \frac{\tau_e(1+p^{-1})}{3}-\frac{(1+p^{-\frac 13})^2}{p}; \\
  g(e,2,p)&:= \frac{(e+1)(1+p^{-\frac 13})^3}{6} -\frac{(1+(-1)^e)(1+p^{-\frac 13})(1+p^{-\frac 23})}{4}  + \frac{\tau_e(1+p^{-1})}{3} .
\end{align*}
\end{lemma}

\begin{proof} 
We may write $ m = \prod_{j=1}^J p_j^{e_j}$ and $ h = \prod_{j=1}^J  p_j^{s_j }$, where
$ p_1 , \ldots , p_J$ are distinct primes and for each $ j $, $ e_j $ and $ s_j  $ are nonnegative integers, but not both zero. Then we see that
\begin{align*}
\sumD \lambda_K (m) \mu_K (h) = & \sumD \prod_{j=1}^J \bigg( \lambda_K \big( p_j ^{e_j} \big) \mu_K \big( p_j^{s_j} \big) \bigg) =  \sum_{\vk} \sum_{ \substack{ K \in \F^{\pm}(X) \\   \vp  : ~type ~T_{\vk}  }} \prod_{j=1}^J \bigg( \lambda_K \big( p_j ^{e_j} \big) \mu_K \big( p_j^{s_j} \big) \bigg) ,
\end{align*} 
where $\vk = (k_1 , \ldots , k_J ) $ runs over $\{ 1, 2, 3, 4, 5 \}^J $ and $\vp = ( p_1 , \ldots , p_J)$. When each $p_j$ has splitting type $T_{k_j}$ in $K$, the values $\lambda_K (p_j^{e_j})$ and $ \mu_K ( p_j^{s_j})$ depend  on $p_j$, $k_j $, $e_j$ and $ s_j$. Define
\begin{equation*}\label{def eta}
 \eta_{1, p_j}(k_j, e_j ) := \lambda_K (p_j^{e_j}) ,  \qquad   \eta_{2, p_j} (k_j, s_j ) :=   \mu_K (p_j^{s_j})  
 \end{equation*}
for each $ j \leq J$ with $p_j $ of splitting type $T_{k_j}$ in $K$, as well as
\begin{equation}\label{def eta vec}
 \eta_{1, \vp} ( \vk, \ve) := \prod_{j=1}^J \eta_{1, p_j} (k_j, e_j ),  \qquad  \eta_{2, \vp} ( \vk, \vs) := \prod_{j=1}^J \eta_{2, p_j} (k_j, s_j ).
 \end{equation}
We see that
\begin{align*}
\sumD \lambda_K (m) \mu_K (h)  & =  \sum_{\vk} \eta_{1, \vp} ( \vk, \ve) \eta_{2,\vp} ( \vk, \vs) \sum_{ \substack{ K \in \F^{\pm}(X) \\   \vp  : ~type ~T_{\vk}  }}  1 \\
 & =  \sum_{\vk} \eta_{1, \vp} ( \vk, \ve) \eta_{2,\vp} ( \vk, \vs)  N^{\pm}_{\vp} ( X , T_{\vk}),
\end{align*}
which by~\eqref{equation TT local 2 vector form} is equal to
\begin{align*}
&  \sum_{\vk} \eta_{1, \vp} ( \vk, \ve) \eta_{2,\vp} ( \vk, \vs) \bigg(C_1^\pm \prod_{j = 1}^J ( x_{p_j} c_{k_j} (p_j ) ) X + C_2^\pm \prod_{j = 1}^J ( y_{p_j} d_{k_j} (p_j ) )  X^{\frac 56} + O_{\eps} \bigg(\prod_{j =1}^J p_j^{\omega}X^{\theta+\eps}   \bigg)   \bigg) \\
= &  C_1^\pm X  \bigg(  \sum_{\vk} \eta_{1, \vp} ( \vk, \ve) \eta_{2,\vp} ( \vk, \vs)   \prod_{j = 1}^J  ( x_{p_j} c_{k_j} (p_j ))   \bigg)       +  C_2^\pm   X^{\frac 56}  \bigg(   \sum_{\vk} \eta_{1, \vp} ( \vk, \ve) \eta_{2,\vp} ( \vk, \vs) \prod_{j = 1}^J    ( y_{p_j} d_{k_j} (p_j ))   \bigg)    \\
& +  O_{\eps} \bigg(  \sum_{\vk} | \eta_{1, \vp} ( \vk, \ve)  \eta_{2,\vp} ( \vk, \vs) |       \prod_{j =1}^J p_j^{\omega}X^{\theta+\eps}   \bigg) .     
\end{align*}
We can change the last three $\vk$-sums into products by \eqref{def eta vec}. Doing so, we obtain that the above is equal to
\begin{align*}
 &C_1^\pm X  \prod_{j =1}^J  \bigg( x_{p_j}  \widetilde f(e_j, s_j, p_j )   \bigg)      +  C_2^\pm   X^{\frac 56} \prod_{j =1}^J  \bigg(  y_{p_j}   \widetilde g(e_j, s_j, p_j )   \bigg)        +  O_{\eps} \bigg(  \prod_{j =1}^J \bigg( p_j^{\omega}    (2e_j +5 )    \bigg)  X^{\theta+\eps}   \bigg)   \\
 =& C_1^\pm X \prod_{p^e\parallel m, p^s \parallel h}\widetilde f(e,s,p)x_p +  C_2^\pm X^{\frac 56}    \prod_{p^e\parallel m, p^s \parallel h}\widetilde g(e,s,p)y_p +  O_{\eps} \bigg(  \prod_{j =1}^J \bigg( p_j^{\omega}   ( 2 e_j + 5 )    \bigg)  X^{\theta+\eps}   \bigg) ,  
\end{align*}
where
$$ \widetilde f( e, s, p) := \sum_{k=1}^5 \eta_{1, p } (k  , e  ) \eta_{2,p   } (k  , s  )    c_{k } (p  ) , \qquad \widetilde g(e,s,p) :=    \sum_{k  =1}^5   \eta_{1, p  } ( k  , e  ) \eta_{2,p   } ( k  , s    )    d_{k } (p  ) .  $$
A straightforward calculation shows that $\widetilde f( e, s, p) = f( e, s, p) $ and $\widetilde g( e, s, p)=g( e, s, p)  $ (see the explicit description of the coefficients in Section~\ref{section background}; note that $ \eta_{2,p}(k,0) = 1 $), and the lemma follows. 
\end{proof}

We now proceed with the estimation of $R_1(\alpha, \gamma ; X )$. Taking into account the two main terms in Lemma \ref{lemma count1}, we expect that
\begin{equation}
R_1(\alpha, \gamma;X)=R_1^M(\alpha, \gamma)+\frac{C_2^\pm}{C_1^\pm}X^{-\frac 16}\Big( 1-\frac{C_2^\pm}{C_1^\pm}X^{-\frac{1}{6}}\Big) (R_1^S(\alpha, \gamma)-R_1^M(\alpha, \gamma))+\text{Error},
\label{equation expansion R_1}
\end{equation}
where 
\begin{align}
\label{equation definition R_1M}
\begin{split}R_1^M(\alpha, \gamma):=&\prod_p \Big( 1 + \sum_{ e\geq 1} \frac{x_p f(e,0,p)}{p^{e(\frac 12+\alpha)}} + \sum_{ e\geq 0} \frac{x_p f(e,1,p)}{p^{e(\frac 12+\alpha) +(\frac 12+\gamma) }} + \sum_{ e\geq 0} \frac{x_p f(e,2,p)}{p^{e(\frac 12+\alpha) + 2(\frac 12+\gamma) }}\Big),\\
R_1^S(\alpha, \gamma):=&\prod_p \Big( 1 + \sum_{ e\geq 1} \frac{y_p g(e,0,p)}{p^{e(\frac 12+\alpha)}} + \sum_{ e\geq 0} \frac{y_p g(e,1,p)}{p^{e(\frac 12+\alpha) +(\frac 12+\gamma) }} + \sum_{ e\geq 0} \frac{y_p g(e,2,p)}{p^{e(\frac 12+\alpha) + 2(\frac 12+\gamma) }}\Big) 
\end{split}
\end{align}
for $ \Re ( \alpha) , \Re ( \gamma) > \frac 12$. Since
\begin{multline*}
R_1^M(\alpha, \gamma) 
= \prod_p \Bigg( 1 + \frac{1}{p^{1+2\alpha}}  - \frac{1}{p^{1+\alpha+\gamma}}\\ + O \Big(\frac{1}{p^{\frac 32 +\Re(\alpha)}}+\frac{1}{p^{\frac 32 +3\Re(\alpha)}}+\frac{1}{p^{\frac 32 +\Re(\gamma)}} + \frac{1}{p^{\frac 32+ \Re ( 2 \alpha + \gamma) }} + \frac{1}{p^{\frac 52 + \Re (3 \alpha+ 2 \gamma) }} \Big) \Bigg),
\end{multline*}
we see that 
\begin{equation}
 A_3(\alpha, \gamma) :=   \frac{\zeta(1+\alpha+\gamma)}{\zeta(1+2\alpha) }  R_1^M(\alpha, \gamma) 
 \label{equation definition A_3}
\end{equation}
 is analytically continued to the region\footnote{To see this, write $\frac{\zeta(1+\alpha+\gamma)}{\zeta(1+2\alpha) }$ as an Euler product, and expand out the triple product in~\eqref{equation definition A_3}. The resulting expression will converge in the stated region.} $ \Re (\alpha), \Re (\gamma)> - \frac 16 $. Similarly, from the estimates
\begin{align*}
& \sum_{ e\geq 1} \frac{y_p g(e,0,p)}{p^{e(\frac 12+\alpha)}} = \frac 1{p^{\frac 56+\alpha}}+\frac 1{p^{1+2\alpha}} +O\Big( \frac 1{p^{\Re(\alpha) + \frac 32}}+\frac 1{p^{2\Re(\alpha) + \frac 43}}+ \frac 1{p^{3\Re(\alpha) + \frac 32}}\Big),\\
&\sum_{ e\geq 0} \frac{y_p g(e,1,p)}{p^{e(\frac 12+\alpha) +(\frac 12+\gamma) }}=-\frac 1{p^{\frac 56+\gamma}} - \frac 1{p^{1+\alpha+\gamma}}+O\Big(\frac 1{p^{\frac 43+\Re(\alpha+\gamma)}} \Big),\\
&\sum_{ e\geq 0} \frac{y_p g(e,2,p)}{p^{e(\frac 12+\alpha) + 2(\frac 12+\gamma) }}  = O \Big(  \frac 1{p^{\frac 32+\Re(\alpha+2\gamma)}} \Big) ,
\end{align*} 
we deduce that
\begin{equation}
A_4(\alpha,\gamma):= \frac{ \zeta(\tfrac 56+\gamma) \zeta(1+\alpha+\gamma)}{\zeta(\tfrac 56+\alpha) \zeta(1+2\alpha) }R_1^S(\alpha, \gamma)
\label{equation definition A_4}
\end{equation}
is analytic in the region $\Re(\alpha),\Re(\gamma)>-\frac 16$.
Note that by their defining product formulas, we have the bounds
\begin{equation}\label{trivial bounds for A3 and A4}
  A_3 ( \alpha, \gamma)=O_\eps (1)  ,\quad  A_4 ( \alpha, \gamma) = O_\eps (1) 
  \end{equation}
for $ \Re ( \alpha) , \Re ( \gamma) \geq - \frac16 + \eps > - \frac16$. Using this notation,~\eqref{equation expansion R_1} takes the form
\begin{multline*}
 R_1(\alpha,\gamma;X) = \frac{\zeta(1+2\alpha)}{\zeta(1+\alpha+\gamma)} \Big( A_3(\alpha,\gamma)+\frac{C_2^\pm}{C_1^\pm} X^{-\frac 16} \Big( 1-\frac{C_2^\pm}{C_1^\pm}X^{-\frac{1}{6}}\Big) \Big( \frac{\zeta(\tfrac 56+\alpha)}{\zeta(\tfrac 56+\gamma)} A_4(\alpha,\gamma) - A_3(\alpha,\gamma) \Big)\Big)\\ +\text{Error}.
\end{multline*}  

The above computation is sufficient in order to obtain a conjectural evaluation of the average~\eqref{ragone}. However, our goal is to evaluate the $1$-level density through the average of $\frac{L'}{L} (\frac 12+r ,f_K ) $; therefore it is necessary to also compute the partial derivative $\frac{\partial }{ \partial \alpha}R_1(\alpha, \gamma;X)|_{\alpha=\gamma=r} $. To do so, we need to make sure that the error term stays small after a differentiation. This is achieved by applying Cauchy's integral formula for the derivative 
$$f'(a)=\frac{1}{2\pi i}\int_{|z-a|=\kappa} \frac{f(z)}{(z-a)^2}dz$$ (valid for all small enough $\kappa>0$), and bounding the integrand using the approximation for $R_1(\alpha,\gamma;X) $ above. As for the main terms, one can differentiate them term by term, and obtain the expected approximation
\begin{multline}
\frac{\partial }{ \partial \alpha}R_1(\alpha, \gamma;X)\Big|_{\alpha=\gamma=r}=A_{3,\alpha}(r,r)+ \frac{\zeta'}{\zeta}(1+2r) A_3(r,r) \\+ \frac{C_2^\pm}{C_1^\pm} X^{-\frac 16}\Big( 1-\frac{C_2^\pm}{C_1^\pm}X^{-\frac{1}{6}}\Big) \Big( A_{4,\alpha}(r,r) +\frac{\zeta'}{\zeta} (\tfrac 56+r) A_4 (r,r) - A_{3,\alpha}(r,r) +\frac{\zeta'}{\zeta}(1+2r) ( A_4(r,r)-A_3(r,r)) \Big) \\
+\text{Error},
\label{equation R_1 calculation}
\end{multline}
where $A_{3,\alpha}(r,r)=\frac{\partial }{ \partial \alpha}A_3(\alpha,\gamma)\Big|_{\alpha=\gamma=r}$ and $A_{4,\alpha}(r,r)=\frac{\partial }{ \partial \alpha}A_4(\alpha,\gamma)\Big|_{\alpha=\gamma=r}.$

Now, from the definition of $f(e,j,p) $ and $g(e,j,p)$ (see Lemma~\ref{lemma count1}) as well as~\eqref{definition theta_e} and~\eqref{definition kappa_e}, we have 
\begin{align*}
 f(1,0,p)+f(0,1,p) & =g(1,0,p)+g(0,1,p)=0,\\
f(e,0,p)+f(e-1,1,p) +f(e-2,2,p) & =g(e,0,p)+g(e-1,1,p) +g(e-2,2,p)=0, \\
 f(e, 0 , p ) - f(e-2 , 2, p ) & = \theta_e + p^{-1}  , \\
   g(e, 0 , p ) - g(e-2 , 2, p ) & =  ( 1 + p^{- \frac13} ) ( \kappa_e ( p ) + p^{-1} + p^{- \frac43} )  .
  \end{align*}
   By the above identities and the definition~\eqref{equation definition R_1M}, we deduce that
$$R_1^M(r,r)=A_3(r,r)=R_1^S(r,r)=A_4(r,r)=1. $$
 It follows that for $\Re(r)>\frac 12$,
 \begin{align*}
& R^M_{1,\alpha}(r,r)  =  \frac{R^M_{1,\alpha}(r,r)}{R^M_1(r,r)} = \frac{ \partial}{\partial \alpha}   \log R^M_{1} ( \alpha, \gamma) \bigg|_{ \alpha=\gamma = r}\\
  &= \sum_p \left(  - \frac{  x_p \log p }{p^{\frac12 + r}}f(1,0,p) - \sum_{ e \geq 2} \frac{ x_p \log p }{ p^{ e( \frac12 + r ) } }  \big( f(e,0,p ) - f(e-2,2,p)\big)  \right)\\
   &\hspace{1cm}+ \sum_p \left(   - \sum_{ e \geq 2} \frac{ x_p \log p }{ p^{ e( \frac12 + r ) } }  (e-1) \big( f( e,0,p) + f(e-1,1,p) + f( e-2,2,p) \big)  \right)   \\
  & =- \sum_p      \sum_{ e \geq 1} \frac{ x_p \log p }{ p^{ e( \frac12 + r ) } }  \Big( \theta_e + \frac1p \Big)
   \end{align*}  
   and
    \begin{align*}
  R^S_{1,\alpha}(r,r)   
  &= \sum_p \left(  - \frac{  y_p \log p }{p^{\frac12 + r}}g(1,0,p) - \sum_{ e \geq 2} \frac{ y_p \log p }{ p^{ e( \frac12 + r ) } }  \big( g(e,0,p ) - g(e-2,2,p)\big)  \right)\\
   &\hspace{1cm}+ \sum_p \left(   - \sum_{ e \geq 2} \frac{ y_p \log p }{ p^{ e( \frac12 + r ) } }  (e-1) \big( g( e,0,p) + g(e-1,1,p) + g( e-2,2,p) \big)  \right)   \\
  &   =- \sum_p      \sum_{ e \geq 1} \frac{y_p \log p }{ p^{ e( \frac12 + r ) } }  \big( 1 + p^{- \frac13} \big) \big( \kappa_e ( p ) + p^{-1} + p^{- \frac43} \big)  \\
  & =- \sum_p      \sum_{ e \geq 1} \frac{  \log p }{ p^{ e( \frac12 + r ) } }   \Big( \beta_e (p)+ x_p \Big( \theta_e+ \frac1p \Big)\Big) ,
   \end{align*}  
   by~\eqref{eqaution definition beta}.
Thus, we have
\begin{align*}
A_{3,\alpha}(r,r)&=R^M_{1,\alpha}(r,r) - \frac{\zeta'}{\zeta}(1+2r)  = -\sum_{p,e\geq 1}  \Big(\theta_e+\frac 1p\Big) \frac{x_p\log p }{p^{e(\frac 12+r)}}- \frac{\zeta'}{\zeta}(1+2r)
\end{align*}
and
\begin{align}
A_{4,\alpha}(r,r)-A_{3,\alpha}(r,r) = -\sum_{p,e\geq 1}   \frac{(\beta_e(p)-p^{-\frac e3}) \log p }{p^{e(\frac 12+r)}},
\end{align}
which are now valid in the extended region $ \Re(r)>0 $. Coming back to~\eqref{equation R_1 calculation}, we deduce that
\begin{align*}\notag 
\frac{\partial }{ \partial \alpha}R_1(\alpha, \gamma;X)\Big|_{\alpha=\gamma=r}  & =A_{3,\alpha}(r,r)+ \frac{\zeta'}{\zeta}(1+2r)  \\&\notag\hspace{1cm}+ \frac{C_2^\pm}{C_1^\pm} X^{-\frac 16}\Big( 1-\frac{C_2^\pm}{C_1^\pm}X^{-\frac{1}{6}}\Big) \Big( A_{4,\alpha}(r,r) - A_{3,\alpha}(r,r)   +\frac{\zeta'}{\zeta} (\tfrac 56+r) \Big) +\text{Error}\\
\notag &=-\sum_{p,e\geq 1}  \Big(\theta_e+\frac 1p\Big) \frac{x_p\log p }{p^{e(\frac 12+r)}} 
-\frac{C_2^\pm}{C_1^\pm} X^{-\frac 16}\Big( 1-\frac{C_2^\pm}{C_1^\pm}X^{-\frac{1}{6}}\Big) \sum_{p,e\geq 1}   \frac{(\beta_e(p)-p^{-\frac e3} )\log p }{p^{e(\frac 12+r)}}\\
&\hspace{1cm}+\frac{C_2^\pm}{C_1^\pm} X^{-\frac 16}\Big( 1-\frac{C_2^\pm}{C_1^\pm}X^{-\frac{1}{6}}\Big)\frac{\zeta'}{\zeta} (\tfrac 56+r)
+\text{Error},
\label{equation contribution R_1 to RC}
\end{align*}
where the second equality is valid in the region $\Re(r)>0$.

We now move to $R_2(\alpha, \gamma ;X).$ We recall that
\begin{equation} 
R_2(\alpha, \gamma;X)=\frac1{N^\pm(X)}\sumD  |D_K|^{-\alpha}\frac{\Ga_{\pm}(\frac 12-\alpha)}{\Ga_{\pm}(\frac 12+\alpha)} \sum_{h,m} \frac{\lambda_K(m)\mu_{K}(h)}{m^{\frac{1}{2}-\alpha}h^{\frac{1}{2}+\gamma}},
\end{equation}
and the Ratios Conjecture recipe tells us that we should replace $\lambda_K(m)\mu_K(h)$ with its average. However, a calculation involving Lemma~\ref{lemma count1} suggests that the terms $ |D_K|^{-\alpha}$ and $\lambda_K(m)\mu_{K}(h)$ have non-negligible covariance. To take this into account, we substitute this step with the use of the following corollary of Lemma~\ref{lemma count1}.

\begin{corollary}
\label{lemma average oscillatory}
Let $m,h\in \mathbb N$, and let $\frac 12 \leq \theta<\frac 56$ and $\omega\geq0$ be such that~\eqref{equation TT local 2 vector form} holds. For $\alpha\in \C$ with $0<\Re(\alpha) < \frac 12$, we have the estimate
\begin{multline*} 
\frac{1}{N^\pm(X)}\sumD |D_K|^{-\alpha} \lambda_K(m)\mu_K(h) 
= \frac{X^{-\alpha}}{1-\alpha} \prod_{p^e\parallel m, p^s \parallel h}f(e,s,p)x_p \\ + X^{-\frac16-\alpha} \Bigg(  \frac{1}{1- \frac{6\alpha}5}   \prod_{p^e\parallel m, p^s \parallel h}g(e,s,p)y_p-   \frac{1}{1-\alpha} \prod_{p^e\parallel m, p^s \parallel h}f(e,s,p)x_p\Bigg) \frac{C_2^\pm}{C_1^\pm} \Big( 1-\frac{C_2^\pm}{C_1^\pm}X^{-\frac{1}{6}}\Big)   \\
+O_{\eps}\bigg((1+|\alpha|)\prod_{ p \mid hm, p^e \parallel m } \big( (2 e + 5)p^{\omega}\big)X^{\theta-1-\Re(\alpha)+\eps}\bigg).
\end{multline*}
\end{corollary}
\begin{proof}
This follows from applying Lemma \ref{lemma count1} and (\ref{equation TT}) to the identity
\begin{multline*}
\sumD |D_K|^{-\alpha} \lambda_K(m)\mu_K(h) =  \int_1^X u^{-\alpha} d\bigg(\sumDu  \lambda_K(m)\mu_K(h)  \bigg) \\
=  X^{-\alpha}  \sumD  \lambda_K(m)\mu_K(h)   + \alpha \int_1^X u^{-\alpha -1 }  \bigg(\sumDu  \lambda_K(m)\mu_K(h)\bigg)    du .
\end{multline*}
\end{proof}

Applying this lemma, we deduce the following heuristic approximation of $R_2 ( \alpha, \gamma ; X)$:
\begin{equation*}\begin{split}
  &  \frac{\Ga_{\pm}(\frac 12-\alpha)}{\Ga_{\pm}(\frac 12+\alpha)} \sum_{h,m} \frac{1}{m^{\frac{1}{2}-\alpha}h^{\frac{1}{2}+\gamma}} \bigg\{ \frac{X^{-\alpha}}{1-\alpha} \prod_{p^e\parallel m, p^s \parallel h}f(e,s,p)x_p\\
& + X^{-\frac 16-\alpha} \frac{C_2^\pm}{C_1^\pm}  \Big( 1-\frac{C_2^\pm}{C_1^\pm}X^{-\frac{1}{6}}\Big)  \Big(  \frac{1}{1-\frac{6\alpha}5}   \prod_{p^e\parallel m, p^s \parallel h}g(e,s,p)y_p-   \frac{1}{1-\alpha} \prod_{p^e\parallel m, p^s \parallel h}f(e,s,p)x_p\Big)  \bigg\}   \\
= & \frac{\Ga_{\pm}(\frac 12-\alpha)}{\Ga_{\pm}(\frac 12+\alpha)}   \bigg\{ X^{-\alpha} \frac{ R_1^M ( - \alpha, \gamma)}{1-\alpha}   + X^{-\frac 16-\alpha}\frac{C_2^\pm}{C_1^\pm}  \Big( 1-\frac{C_2^\pm}{C_1^\pm}X^{-\frac{1}{6}}\Big)   \Big(  \frac{R_1^S ( - \alpha, \gamma) }{1-\frac{6\alpha}5}    -   \frac{R_1^M ( - \alpha, \gamma)}{1-\alpha}  \Big)  \bigg\}  
 \\
= & \frac{\Ga_{\pm}(\frac 12-\alpha)}{\Ga_{\pm}(\frac 12+\alpha)}  \frac{\zeta(1-2\alpha)}{  \zeta( 1 - \alpha + \gamma) }  \bigg\{  X^{-\alpha} \frac{A_3  ( - \alpha, \gamma)}{1-\alpha}    \\
&  + X^{-\frac 16-\alpha}\frac{C_2^\pm}{C_1^\pm}  \Big( 1-\frac{C_2^\pm}{C_1^\pm}X^{-\frac{1}{6}}\Big)   \Big(  \frac{A_4 ( - \alpha, \gamma)}{1- \frac{6\alpha}5}  \frac{  \zeta(\frac 56-\alpha)}{ \zeta( \frac 56+\gamma)}   -   \frac{A_3 ( - \alpha, \gamma) }{1-\alpha} \Big)  \bigg\}   .
\end{split}\end{equation*}
If $\Re( r) $ is positive and small enough, then we expect that 
\begin{equation*}\begin{split}
  \frac{\partial }{ \partial \alpha}R_2(\alpha, \gamma;X)\Big|_{\alpha=\gamma=r}  = &  - \frac{\Ga_{\pm}(\frac 12-r)}{\Ga_{\pm}(\frac 12+r)}   \zeta(1-2r)   \bigg\{ X^{-r}  \frac{A_3  ( - r,r)}{1-r}    \\
&   + X^{-\frac 16-r} \frac{C_2^\pm}{C_1^\pm}  \Big( 1-\frac{C_2^\pm}{C_1^\pm}X^{-\frac{1}{6}}\Big)   \Big(   \frac{\zeta(\tfrac 56-r)}{\zeta(\tfrac 56+r)}  \frac{A_4(-r,r)}{ 1- \frac{6r}5}  -  \frac{ A_3( -r, r)}{1-r} \Big)  \bigg\}+\text{Error}  .
\end{split}\end{equation*}
We arrive at the following conjecture.

\begin{conjecture} \label{ratios-thm}
Let $\frac 12 \leq \theta<\frac 56$ and $\omega\geq0$ be such that~\eqref{equation TT local 2 vector form} holds. There exists $0<\delta<\frac 1{6}$ such that for any fixed $\eps>0$ and for  $r\in \mathbb C$ with $ \frac 1{L}  \ll \Re(r)  <\delta $ and
$|r|\leq X^{\frac \eps 2}$, 
\begin{multline} \label{2nd conj eqn}
\frac 1{N^\pm(X)}\sumD \frac{L'(\frac 12+r,f_K)}{L(\frac 12+r,f_K)}\\
= -\sum_{p,e\geq 1}  \Big(\theta_e+\frac 1p\Big) \frac{x_p\log p }{p^{e(\frac 12+r)}} 
-\frac{C_2^\pm}{C_1^\pm} X^{-\frac 16} \Big( 1-\frac{C_2^\pm}{C_1^\pm}X^{-\frac{1}{6}}\Big) \sum_{p,e\geq 1}   \frac{(\beta_e(p)-p^{-\frac e3}) \log p }{p^{e(\frac 12+r)}} \\
+\frac{C_2^\pm}{C_1^\pm} X^{-\frac 16}\Big( 1-\frac{C_2^\pm}{C_1^\pm}X^{-\frac{1}{6}}\Big)  \frac{\zeta'}{\zeta}(\tfrac 56+r)
-  X^{-r}  \frac{\Ga_{\pm}(\frac 12-r)}{\Ga_{\pm}(\frac 12+r)}\zeta(1-2r)  \frac{A_3( -r, r)}{1-r} \\
- \frac{C_2^\pm }{C_1^\pm  } X^{-r-\frac 16} \Big( 1-\frac{C_2^\pm}{C_1^\pm}X^{-\frac{1}{6}}\Big)   \frac{\Ga_{\pm}(\frac 12-r)}{\Ga_{\pm}(\frac 12+r)}\zeta(1-2r) \Big(  \frac{\zeta(\tfrac 56-r)}{\zeta(\tfrac 56+r)}  \frac{A_4(-r,r)}{ 1- \frac{6r}5}  -  \frac{ A_3( -r, r)}{1-r} \Big) \\
+O_\eps(X^{\theta-1+\eps}).
\end{multline}
Note that the two sums on the right-hand side are absolutely convergent. 
\end{conjecture}

Traditionally, when applying the Ratios Conjecture recipe, one has to restrict the real part of the variable $r$ to small enough positive values. For example, in the family of quadratic Dirichlet $L$-functions~\cite{CS,FPS3}, one requires that $\frac1{\log X}\ll\Re(r)<\frac 14$. This ensures that one is far enough from a pole for the expression in the right-hand side. In the current situation, we will see that the term involving $X^{-r-\frac 16}$ has a pole at $s=\frac 16$.

\begin{proposition}
\label{proposition RC}
Assume Conjecture~\ref{ratios-thm} and the Riemann Hypothesis for $\zeta_K(s)$ for all $K\in \F^{\pm}(X)$, and let $\phi$ be a real even Schwartz function such that $ \widehat \phi$ is compactly supported. For any constant $ 0< c < \frac16$, we have that
\begin{multline} 
\frac 1{N^\pm(X)}\sumD \sum_{\gamma_K}\phi \Big(\frac{L\gamma_K}{2\pi }\Big)=\widehat \phi(0)\Big(1+ \frac{\log(4 \pi^2 e)}{L} -\frac{C_2^\pm}{5C_1^\pm} \frac{X^{-\frac 16}}{L} 
+ \frac{(C_2^\pm)^2 }{5(C_1^\pm)^2 } \frac{X^{-\frac 13}}{L} \Big)\\+  \frac1{\pi}\int_{-\infty}^{\infty}\phi\left(\frac{Lr}{2\pi}\right)\Re\Big(\frac{\Ga'_{\pm}}{\Ga_{\pm}}(\tfrac12+ir)\Big)dr   -\frac{2}{L}\sum_{p,e}\frac{x_p\log p}{p^{\frac e2}}\widehat\phi\left(\frac{\log p^e}{L}\right) (\theta_e+\tfrac 1p)\\
-\frac{2C_2^\pm X^{-\frac 16}}{C_1^\pm L}\Big( 1-\frac{C_2^\pm}{C_1^\pm}X^{-\frac{1}{6}}\Big) \sum_{p,e}\frac{\log p}{p^{\frac e2}}\widehat\phi\left(\frac{\log p^e}{L}\right) (\beta_e(p)-p^{-\frac e3})\\
-\frac{1}{\pi i} \int_{(c )} \phi \Big(\frac{Ls}{2\pi i}\Big) \Big\{ -\frac{C_2^\pm}{C_1^\pm} X^{-\frac 16}\Big( 1-\frac{C_2^\pm}{C_1^\pm}X^{-\frac{1}{6}}\Big) \frac{\zeta'}{\zeta}(\tfrac 56+s) +  X^{-s}  \frac{\Ga_{\pm}(\frac 12-s)}{\Ga_{\pm}(\frac 12+s)}\zeta(1-2s)  \frac{A_3( -s, s) }{1-s} \\
+\frac{C_2^\pm }{C_1^\pm} X^{-s-\frac 16}\Big( 1-\frac{C_2^\pm}{C_1^\pm}X^{-\frac{1}{6}}\Big)  \frac{\Ga_{\pm}(\frac 12-s)}{\Ga_{\pm}(\frac 12+s)}\zeta(1-2s) \Big(  \frac{\zeta(\tfrac 56-s)}{\zeta(\tfrac 56+s)}  \frac{A_4(-s,s)}{ 1- \frac{6s}5}  -  \frac{ A_3( -s, s)}{1-s} \Big) \Big\}ds 
\\ + O_\eps(X^{\theta-1+\eps}).
\label{equation proposition RC}
\end{multline}
\end{proposition}

\begin{proof}
By the residue theorem, we have the identity
\begin{align}
\frac 1{N^\pm(X)}\sumD \D_\phi(K)= \frac{1}{2 \pi i} \left( \int_{(\frac1L)} - \int_{(-\frac1L)} \right) \frac 1{N^\pm(X)}\sumD \frac{L'(s+\frac 12,f_K)}{L(s+\frac 12,f_K)}\phi \Big(\frac{Ls}{2\pi i}\Big)ds.
\label{equation residue theorem}
\end{align}
Under Conjecture~\ref{ratios-thm} and well-known arguments (see e.g.~\cite[Section 3.2]{FPS3}), the part of this sum involving the first integral is equal to  
\begin{multline*}
-\frac{1}{2\pi  i} \int_{(\frac1L)} \phi \Big(\frac{Ls}{2\pi i}\Big) \bigg\{\sum_{p,e\geq 1} \Big(\theta_e+\frac 1p\Big) \frac{x_p\log p }{p^{e(\frac 12+s)}} 
+\frac{C_2^\pm}{C_1^\pm} X^{-\frac 16} \Big( 1-\frac{C_2^\pm}{C_1^\pm}X^{-\frac{1}{6}}\Big) \sum_{p,e\geq 1}   \frac{(\beta_e(p)-p^{-\frac e3}) \log p }{p^{e(\frac 12+s)}} \\-\frac{C_2^\pm}{C_1^\pm} X^{-\frac 16}\Big( 1-\frac{C_2^\pm}{C_1^\pm}X^{-\frac{1}{6}}\Big) \frac{\zeta'}{\zeta}(\tfrac 56+s) + X^{-s}   \frac{\Ga_{\pm}(\frac 12-s)}{\Ga_{\pm}(\frac 12+s)}\zeta(1-2s)  \frac{A_3( -s, s) }{1-s} \\
+\frac{C_2^\pm }{C_1^\pm}X^{-s-\frac 16}\Big( 1-\frac{C_2^\pm}{C_1^\pm}X^{-\frac{1}{6}}\Big)   \frac{\Ga_{\pm}(\frac 12-s)}{\Ga_{\pm}(\frac 12+s)}\zeta(1-2s) \Big(  \frac{\zeta(\tfrac 56-s)}{\zeta(\tfrac 56+s)}  \frac{A_4(-s,s)}{ 1- \frac{6s}5}  -  \frac{ A_3( -s, s)}{1-s} \Big) \bigg\} ds \\+O_\eps(X^{\theta-1+\eps}),
\end{multline*}
where we used the bounds \eqref{trivial bounds for A3 and A4} and 
\begin{align}
 \phi\Big(\frac{Ls}{2\pi i}\Big)  = \frac { (-1)^\ell} {L^\ell s^\ell }\int_{\mathbb R} e^{ L \Re(s) x} e^{ i L \Im(s) x} \widehat \phi^{(\ell)}(x) dx  \ll_\ell  \frac{e^{L|\Re(s)| \sup({\rm supp}(\widehat \phi))  } }{L^\ell |s|^\ell } 
 \label{equation decay phi complex}
\end{align} 
for every integer $\ell>0 $, which is decaying on the line $\Re(s) = \frac 1L$.   We may also shift the contour of integration to the line $\Re(s)=c$ with $ 0 < c < \frac16$.

For the second integral in~\eqref{equation residue theorem} (over the line $\Re(s)=-\frac{1}{L}$), we treat it as follows. By the functional equation \eqref{equation functional equation}, we have
\begin{align*}
- \frac{1}{2 \pi i} &    \int_{(-\frac1L)}   \frac 1{N^\pm(X)}\sumD \frac{L'(s+\frac 12,f_K)}{L(s+\frac 12,f_K)}\phi \Big(\frac{Ls}{2\pi i}\Big)ds \\
 = &  \frac{1}{2 \pi i}  \int_{(\frac1L)}   \frac 1{N^\pm(X)}\sumD \frac{L'(s+\frac 12,f_K)}{L(s+\frac 12,f_K)}\phi \Big(\frac{Ls}{2\pi i}\Big)ds  \\
 & + \frac{1}{2 \pi i}      \int_{(-\frac1L)}   \frac 1{N^\pm(X)}\sumD \left( \log |D_K| + \frac{ \Gamma'_{\pm}}{\Gamma_\pm} ( \tfrac12 + s  )+ \frac{ \Gamma'_{\pm}}{\Gamma_\pm} ( \tfrac12 - s  )   \right)\phi \Big(\frac{Ls}{2\pi i}\Big)ds. 
\end{align*}
The first integral on the right-hand side is identically equal to the integral that was just evaluated in the first part of this proof.
As for the second, by shifting the contour to the line $\Re (s)=0$, we find that it equals
\begin{align*}
  &  \left(  \frac 1{N^\pm(X)}\sumD   \log |D_K|  \right)  \frac{1}{2 \pi i}      \int_{(0)} \phi \Big(\frac{Ls}{2\pi i}\Big)ds    +  \frac{1}{2 \pi i}      \int_{(0)} \left( \frac{ \Gamma'_{\pm}}{\Gamma_\pm} ( \tfrac12 + s  )+ \frac{ \Gamma'_{\pm}}{\Gamma_\pm} ( \tfrac12 - s  )   \right)  \phi \Big(\frac{Ls}{2\pi i}\Big)ds \\
  &  = \left(  \frac 1{N^\pm(X)}\sumD   \log |D_K|  \right) \frac{  \widehat{\phi} (0)}{L} + \frac{1}{ \pi} \int_{- \infty}^\infty \phi \Big(\frac{Lr}{2\pi  }\Big) \Re \left(  \frac{ \Gamma'_{\pm}}{\Gamma_\pm} ( \tfrac12 + ir  ) \right)dr .
\end{align*}
By applying Lemma \ref{lemma average log} to the first term, we find the leading terms on the right-hand side of \eqref{equation proposition RC}.

 
 Finally, by absolute convergence we have the identity
\begin{align*}
\frac{1}{2\pi i} \int_{(c)}\phi \Big(\frac{Ls}{2\pi i}\Big)\sum_{p,e\geq 1} \Big(\theta_e+\frac 1p\Big) \frac{x_p\log p }{p^{e(\frac 12+s)}} ds
&= \sum_{p,e\geq 1} \Big(\theta_e+\frac 1p\Big) \frac{x_p\log p } {p^{\frac e2}} \frac 1{2\pi i} \int_{(c)}\phi \Big(\frac{Ls}{2\pi i}\Big) p^{-es}ds\\
&=\frac{1}{L}\sum_{p,e\geq 1} \Big(\theta_e+\frac 1p\Big) \frac{x_p\log p }{p^{\frac e2}}\widehat \phi\Big( \frac{e\log p}{L} \Big),
\end{align*}
since the contour of the inner integral can be shifted to the line $\Re(s)=0$. The same argument works for the term involving $\beta_e(p)-p^{-\frac e3}$. Hence, the proposition follows.
\end{proof}

\section{Analytic continuation of $A_3 ( -s, s ) $ and $A_4 ( -s , s )$}

The goal of this section is to prove Theorem \ref{theorem RC}. To do so, we will need to estimate some of the terms in \eqref{equation proposition RC}, namely 
\begin{equation}\label{J X def}\begin{split}
J^\pm & (X) :=    \frac{2C_2^\pm X^{-\frac 16}}{C_1^\pm L}\Big( 1-\frac{C_2^\pm}{C_1^\pm}X^{-\frac{1}{6}}\Big) \sum_{p,e}\frac{\log p}{p^{\frac {5e}{6}}}\widehat\phi\left(\frac{\log p^e}{L}\right)  \\
& -\frac{1}{\pi i} \int_{(c )} \phi \Big(\frac{Ls}{2\pi i}\Big) \Big\{ -\frac{C_2^\pm}{C_1^\pm} X^{-\frac 16}\Big( 1-\frac{C_2^\pm}{C_1^\pm}X^{-\frac{1}{6}}\Big) \frac{\zeta'}{\zeta}(\tfrac 56+s) +  X^{-s}   \frac{\Ga_{\pm}(\frac 12-s)}{\Ga_{\pm}(\frac 12+s)}\zeta(1-2s) \frac{ A_3( -s, s) }{1-s}\\
& +\frac{C_2^\pm }{C_1^\pm} X^{-s-\frac 16}\Big( 1-\frac{C_2^\pm}{C_1^\pm}X^{-\frac{1}{6}}\Big)  \frac{\Ga_{\pm}(\frac 12-s)}{\Ga_{\pm}(\frac 12+s)}\zeta(1-2s) \Big(   \frac{\zeta(\tfrac 56-s)}{\zeta(\tfrac 56+s)}  \frac{A_4(-s,s)}{ 1- \frac{6s}5}  -  \frac{ A_3( -s, s)}{1-s} \Big) \Big\}ds,  
\end{split}\end{equation}
for $ 0 < c < \frac16 $. The idea is to provide an analytic continuation to the Dirichlet series $A_3(-s,s)$ and $A_4(-s,s)$ in the strip $ 0<\Re(s) < \frac 12$, and to shift the contour of integration to the right. 
 
\begin{lemma}
\label{lemma A_3 analytic continuation}
The product formula
\begin{equation}\label{A3 prod form}
  A_3 ( -s, s) =  \zeta(3)  \zeta( \tfrac32 - 3s )  \prod_p  \bigg( 1   - \frac{1}{  p^{\frac32 +s}} + \frac{1}{p^{\frac52 - s}} - \frac{1}{p^{\frac52 -3s}}   - \frac{1}{p^{3-4s}} + \frac{1}{p^{\frac92 - 5s}} \bigg)
  \end{equation}
provides an analytic continuation of $A_3 (-s, s)$ to $|\Re ( s )| < \frac12$ except for a simple pole at $s= \frac16$ with residue $$-\frac{\zeta(3)}{3\zeta(\frac53)\zeta(2)}. $$
\end{lemma}

\begin{proof}
From~\eqref{equation definition R_1M} and~\eqref{equation definition A_3}, we see that in the region $ | \Re (s) | < \frac 16$,
\begin{align} \notag
 A_3 ( -s, s) & = \prod_p \bigg( 1 - \frac1{p^3} \bigg)^{-1} \bigg( 1 - \frac1{p^{1-2s}}\bigg) \\  &\hspace{1.5cm}\times\bigg( 1+ \frac1p + \frac1{p^2} +   \sum_{e \geq 1 }  \frac{ f(e,0,p)}{p^{e(\frac 12-s)}} +   \sum_{ e \geq 0 }  \frac{f(e,1,p)}{p^{e(\frac 12-s)+\frac12+s}} +  \sum_{e \geq 0 }  \frac{f(e,2,p)}{ p^{e(\frac 12-s)+1+2s }} \bigg)  \notag \\
 & = \zeta(3)  \prod_p   \bigg( 1 - \frac1{p^{1-2s}}\bigg) \bigg(   \frac1{p^2} +   \sum_{e \geq 0 }  \frac{1}{p^{e(\frac 12-s)}}  \bigg(  f(e,0,p)  +   \frac{f(e,1,p)}{p^{\frac 12+s}} +   \frac{f(e,2,p)}{ p^{1+2s }}\bigg) \bigg). \label{equation A3 section 5}
\end{align}
The sum over $e\geq 0$ on the right-hand side is equal to
\begin{multline}
\label{equation geometric series}
     \frac16 \bigg( 1 - \frac{1}{ p^{\frac 12+s}} \bigg)^2  \sum_{e \geq 0 } (e+1) \frac{1}{p^{e(\frac 12-s)}}   +  \frac12 \bigg( 1 - \frac{1}{p^{1+2s}} \bigg) \sum_{e \geq 0 } \frac{ 1+(-1)^e}{2}  \frac{1}{p^{e(\frac 12-s)}} \\
   + \frac13 \bigg( 1 + \frac{1}{p^{\frac 12+s}}+ \frac{1}{p^{1+2s}} \bigg) \sum_{e \geq 0 }\tau_e \frac{1}{p^{e(\frac 12-s)}}  + \frac1p \bigg(  1- \frac{1}{p^{\frac 12+s}}\bigg) \sum_{e \geq 0 } \frac{1}{p^{e(\frac 12-s)}}\\  
    =  \frac16 \cdot \frac{ \Big(1 - \frac{1}{ p^{\frac 12+s}} \Big)^2}{ \Big(1 - \frac{1}{ p^{\frac 12-s}}\Big)^2}  +  \frac12 \cdot \frac{  1 - \frac{1}{p^{1+2s}} }{ 1 - \frac{1}{p^{1-2s}} }  + \frac13 \cdot \frac{ 1 + \frac{1}{p^{\frac 12+s}}+ \frac{1}{p^{1+2s}}}{ 1 + \frac{1}{p^{\frac 12-s}}+ \frac{1}{p^{1-2s}}  }   + \frac1p  \cdot \frac{  1- \frac{1}{p^{\frac 12+s}}}{  1- \frac{1}{p^{\frac 12-s}} }.
\end{multline}
Here, we have used geometric sum identities, e.g.,
\begin{align*}
\sum_{k=0}^\infty \tau_k x^k & = \sum_{k=0}^\infty x^{3k} - \sum_{k=0}^\infty x^{3k+1} =  \frac{1-x}{1-x^3 } = \frac{1}{1+x+x^2 }  \qquad (|x| < 1).
\end{align*}
Inserting the expression~\eqref{equation geometric series} in~\eqref{equation A3 section 5} and simplifying, we obtain the identity 
\begin{equation*}
  A_3 ( -s, s) =  \zeta(3)  \zeta( \tfrac32 - 3s )  \prod_p  \bigg( 1   - \frac{1}{  p^{\frac 32+s}}   + \frac{1}{p^{\frac 52 - s}} - \frac{1}{p^{\frac 52-3s}}   - \frac{1}{p^{3-4s}} + \frac{1}{p^{\frac 92 - 5s}} \bigg) 
  \end{equation*}
 in the region $|\Re(s)|<1/6$. Now, this clearly extends to 
 $  |  \Re ( s ) |  < 1/2$ except for a simple pole at $ s= 1/6 $ with residue equal to
$$ -\frac{\zeta(3)}{3}  \prod_p  \big( 1 - p^{-\frac53} - p^{-2} + p^{-\frac{11}{3}}\big)  = - \frac{ \zeta(3)}{3} \frac{ 1}{ \zeta(\frac53) \zeta(2)} , $$
as desired.
\end{proof}

\begin{lemma}
\label{lemma A_4 analytic continuation}
Assuming RH, the function $A_4(-s,s)$ admits an analytic continuation to the region $|\Re(s)|< \frac 12$, except for a double pole at $s=\frac 16$. Furthermore, for any $0<\eps<\frac 14$ and in the region $|\Re(s)|< \frac 12-\eps$, we have the bound
$$ A_4(-s,s)  \ll_\eps (|\Im(s)|+1)^{ \frac23} . $$  
\end{lemma}

\begin{proof}
By~\eqref{equation definition A_4} and~\eqref{equation definition R_1M}, for $ | \Re (s) | < \frac16$ we have that
\begin{align*}
 A_4 ( -s, s)   = & \prod_p   \frac{ \Big(1- \frac1{p^{1-2s}}\Big)\Big(  1-   \frac1{p^{\frac56-s}}\Big)\Big( 1 - \frac1{p^{\frac13}} \Big)  }{ \Big(1- \frac1{p^2}\Big)\Big( 1- \frac1{p^{\frac56+s}}\Big) \Big(1- \frac1{p^{\frac53}}\Big) }    \Bigg(  \frac1{p^2} \Big( 1+ \frac1{p^{\frac13}}\Big)  +   \sum_{e \geq 0 }  \frac{     g(e,0,p)  +   \frac{g(e,1,p)}{p^{\frac12+s}} +   \frac{g(e,2,p)}{ p^{1+2s }} }{p^{e(\frac 12-s)}}  \Bigg),
\end{align*}
since $  y_p^{-1} - g(0,0,p) =   \frac1{p^2} \Big( 1+  \frac1{p^{\frac13}}\Big)   $. Recalling the definition of $g(e,j,p)$ (see Lemma~\ref{lemma count1}), a straightforward evaluation of the infinite sum over $e\geq 0$ yields the expression
\begin{multline*}
 A_4 ( -s, s)    
  =  \zeta(2) \zeta(\tfrac53) \prod_p   \frac{ \Big(1- \frac1{p^{1-2s}}\Big)\Big(  1-   \frac1{p^{\frac56 -s}}\Big)\Big( 1 - \frac1{p^{\frac13}}\Big)  }{  \Big( 1- \frac1{p^{\frac56 +s}}\Big)   }    \Bigg(   \frac{ \Big(1+\frac{1}{p^{ \frac13}} \Big)^3\Big(  1 - \frac{1}{ p^{\frac12+s}} \Big)^2 }{ 6    \Big(   1 - \frac{1}{ p^{\frac12-s}}   \Big)^2}  \\
  +  \frac{ \Big(1+\frac{1}{p^{ \frac13}} \Big)\Big(1+\frac{1}{p^{ \frac23}} \Big)  \Big( 1 - \frac{1}{p^{1+2s}}\Big)   }{2 \Big( 1 - \frac{1}{p^{1-2s}}\Big)}   
  + \frac{ \Big( 1+\frac1p \Big)     \Big( 1 + \frac{1}{p^{\frac12 +s}}+ \frac{1}{p^{1+2s}}\Big)   }{ 3 \Big( 1 + \frac{1}{p^{\frac12-s}} + \frac{1}{p^{1-2s}}  \Big) }
  + \frac{\Big(1+\frac{1}{p^{ \frac13}} \Big)^2  \Big( 1- \frac{1}{p^{\frac12+s}}\Big) }{ p\Big(  1- \frac{1}{p^{\frac12-s}}\Big)  } +  \frac{  1+\frac{1}{p^{ \frac13}} }{p^2}     \Bigg).
\end{multline*}  
Isolating the "divergent terms" leads us to the identity
\begin{align*}
A_4 ( -s, s)     
 = & \zeta(2) \zeta(\tfrac53) \prod_p (D_{4,p,1}(s) + A_{4,p,1}(s) ),  
\end{align*}
where
\begin{multline*}
D_{4,p,1}(s):=   \frac{   1-   \frac1{p^{\frac56-s}}   }{    1- \frac1{p^{\frac56+s}}     }    \Bigg(   \frac{ \Big(1+ \frac2{p^{\frac13}}  - \frac2p   \Big)  \Big(1 - \frac{1}{ p^{\frac12+s}} \Big)^2 \Big(1 + \frac1{p^{\frac12-s}}\Big)  }{ 6\Big( 1 - \frac{1}{ p^{\frac12-s}}\Big)}     \\
+  \frac{    1 - \frac{1}{p^{1+2s}}     }{2}  + \frac{\Big(1   - \frac1{p^{\frac13}} +\frac1p   \Big)\Big( 1 + \frac{1}{p^{\frac12+s}}+ \frac{1}{p^{1+2s}} \Big) \Big(1- \frac1{p^{1-2s}}\Big)  }{3\Big( 1 + \frac{1}{p^{\frac12-s}} + \frac{1}{p^{1-2s}} \Big) }     + \frac1p    \Bigg)
\end{multline*}
and
\begin{align*}
A_{4,p,1}(s) := &  \frac{   1-   \frac1{p^{\frac56-s}}   }{    1- \frac1{p^{\frac56+s}}     }    \Bigg(  -  \frac{    \Big(1 - \frac{1}{ p^{\frac12+s}} \Big)^2 \Big(1 + \frac1{p^{\frac12-s}}\Big)  }{ 6 p^{\frac43} \Big( 1 - \frac{1}{ p^{\frac12-s}}\Big)}       -  \frac{    1 - \frac{1}{p^{1+2s}}     }{2p^{\frac43}}  -  \frac{ \Big( 1 + \frac{1}{p^{\frac12+s}}+\frac{1}{p^{1+2s}} \Big) \Big(1- \frac1{p^{1-2s}}\Big)  }{3 p^{\frac43} \Big( 1 + \frac{1}{p^{\frac12-s}} + \frac{1}{p^{1-2s}} \Big) }  \\
& \quad    + \frac{\Big(1+ \frac1{p^{\frac13}}  - \frac1{p^{\frac23}}- \frac1p \Big) \Big( 1- \frac{1}{p^{\frac12+s}}\Big)\Big(  1+ \frac{1}{p^{\frac12-s}} \Big) -1  }{p}  +  \frac1{p^2} \Big( 1- \frac1{p^{\frac23}}\Big)  \Big(1- \frac1{p^{1-2s}}\Big)   \Bigg).
\end{align*}
The term $A_{4,p,1}(s)$ is "small" for $ | \Re (s)|< \frac12$, hence we will concentrate our attention on $D_{4,p,1}(s)$. We see that
$$
D_{4,p,1}(s)
 =  \frac{   1-   \frac1{p^{\frac56-s}}   }{    1- \frac1{p^{\frac56+s}}     }     D_{4,p,2}(s)
 + \frac1p +  A_{4,p,2}(s) ,
$$
where
$$ D_{4,p,2}(s):=    \frac{ \Big(1+ \frac2{p^{\frac13}}  \Big)  \Big(1 - \frac{1}{ p^{\frac12+s}} \Big)^2 \Big(1 + \frac1{p^{\frac12-s}}\Big)  }{ 6\Big( 1 - \frac{1}{ p^{\frac12-s}}\Big)}   
 +  \frac{    1 - \frac{1}{p^{1+2s}}     }{2}  \\ + \frac{\Big(1   - \frac1{p^{\frac13}}    \Big)\Big( 1 + \frac{1}{p^{\frac12+s}}+ \frac{1}{p^{1+2s}} \Big) \Big(1- \frac1{p^{1-2s}}\Big)  }{3\Big( 1 + \frac{1}{p^{\frac12-s}} + \frac{1}{p^{1-2s}} \Big) }    $$
and
\begin{align*}
A_{4,p,2}(s):=    \frac{ \Big(  1-   \frac1{p^{\frac56-s}} \Big)  }{ p \Big(   1- \frac1{p^{\frac56+s}}   \Big)  }      \Bigg(  & -   \frac{  \Big(1 - \frac{1}{ p^{\frac12+s}} \Big)^2 \Big(1 + \frac1{p^{\frac12-s}}\Big)  }{ 3\Big( 1 - \frac{1}{ p^{\frac12 -s}}\Big)}   
+ \frac{ \Big( 1 + \frac{1}{p^{\frac12+s}}+ \frac{1}{p^{1+2s}} \Big) \Big(1- \frac1{p^{1-2s}}\Big)  }{3\Big( 1 + \frac{1}{p^{\frac12-s}} + \frac{1}{p^{1-2s}} \Big) }       +1  \Bigg) - \frac1p,
\end{align*}
which is also "small". Taking common denominators and expanding out shows that
$$ D_{4,p,2}(s)=\frac{1}{  1-  \frac{1}{p^{\frac32 -3s}} } \bigg( 1 - \frac1p        - \frac{1}{p^{\frac56+s}} + \frac{1}{p^{\frac56-s}}   + \frac1{p^{\frac43-2s}} +  A_{4,p,3}(s) \bigg),$$
where
$$ A_{4,p,3}(s) :=      - \frac{1}{p^{\frac32-s}} + \frac{1}{ p^{\frac52-s}}         - \frac1{p^{\frac43}} + \frac{1}{p^{\frac{11}6+s}}   - \frac1{p^{\frac{11}6 -s}} + \frac1{p^{\frac73} } - \frac{1}{p^{\frac73-2s}} $$
is "small". More precisely, for $ | \Re (s) |\leq \frac12 - \eps < \frac12 $ and $ j = 1,2,3$,  we have the bound 
$ A_{4,p,j } (s)  = O_\eps  \Big( \frac{1}{p^{1+\eps}}\Big)  $. Therefore,
\begin{align}
\label{equation before A4tilde}
A_4 ( -s, s)    
=  \zeta(2) \zeta(\tfrac53)  \zeta( \tfrac32-3s) \widetilde{A}_4 (s)  \prod_p \Bigg(    \frac{\Big(  1-   \frac1{p^{\frac56-s}}\Big)  }{  \Big( 1- \frac1{p^{\frac56+s}}\Big)   }        \bigg( 1       - \frac{1}{p^{\frac56+s}} + \frac{1}{p^{\frac56-s}}   + \frac1{p^{\frac43-2s}}        \bigg)   \Bigg) ,
\end{align}
where
\begin{align*}
 \widetilde{A}_4 (s) := &     \prod_p \left(  1 + \frac{ \frac1p \bigg(  1-  \frac{1}{p^{\frac32 -3s}} -  \frac{   1-    p^{-\frac56+s}   }{   1-  p^{- \frac56-s}    }    \bigg)     +    \frac{   1-    p^{-\frac56+s}   }{   1-  p^{- \frac56-s}    }         A_{4,p,3}(s)      +  \Big( 1-  \frac{1}{p^{\frac32-3s}}\Big)(A_{4,p,2}(s) + A_{4,p,1}(s)) }{     \frac{   1-    p^{-\frac56+s}   }{   1-  p^{- \frac56-s}    }        \Big( 1       - \frac{1}{p^{\frac56+s}} + \frac{1}{p^{\frac56-s}}   + \frac1{p^{\frac43-2s}}    \Big)      } \right)
\end{align*}
is absolutely convergent for $|\Re (s) | < \frac12 $. Hence, the final step is to find a meromorphic continuation for the infinite product on the right-hand side of~\eqref{equation before A4tilde}, which we will denote by $D_3(s)$. However, it is straightforward to show that
\begin{equation}\label{definition D3s} 
A_{4,4}(s):= D_3(s)  \frac{ \zeta(\tfrac83-4s) \zeta(\tfrac53-2s) \zeta(\tfrac{13}6-3s) } { \zeta(\tfrac43-2s) \zeta(\tfrac{13}6-s)}  
\end{equation}
converges absolutely for $|\Re(s)|<\frac 12$. This finishes the proof of the first claim in the lemma.

Finally, the growth estimate  
$$ A_4(-s,s)\ll_\eps (|\Im(s)|+1)^{\eps} | \zeta( \tfrac32 - 3s ) \zeta( \tfrac43 - 2s)| \ll_\eps (|\Im(s)|+1)^{ \frac23}  $$  
follows from~\eqref{equation before A4tilde},~\eqref{definition D3s}, as well as~\cite[Theorems 13.18 and 13.23]{MV} and the functional equation for $\zeta(s)$.
\end{proof}

Now that we have a meromorphic continuation of $A_4(-s,s)$, we will calculate the leading Laurent coefficient at $s=\frac 16$.

\begin{lemma}\label{lemma A4 limit at 16}
We have the formula
$$ \lim_{s\rightarrow \frac 16} (s-\tfrac 16)^2 A_4(-s,s) = \frac16  \frac{ \zeta(2)\zeta(\tfrac 53)}{\zeta( \tfrac43)} \prod_p \Big(1-\frac 1{p^{\frac 23}}\Big)^2   \Big(1- \frac1p \Big)\Big(1+\frac 2{p^{\frac 23}}+\frac 1p+\frac 1{p^{\frac 43}}\Big) . $$
\end{lemma}

\begin{proof}
By Lemma \ref{lemma A_4 analytic continuation}, $A_4 ( -s, s ) $ has a double pole at $ s= \frac16 $. Moreover, by \eqref{equation before A4tilde} and \eqref{definition D3s} we find that $  \frac{ A_4 ( -s, s)}{ \zeta( \frac32 - 3s ) \zeta ( \frac43 - 2s ) } $ has a convergent Euler product in the region $|\Re(s)| < \frac 13$ (this allows us to interchange the order of the limit and the product in the calculation below), so that 
\begin{align*}
\lim_{s \to \frac16 }  & (s-\tfrac16 )^2   A_4 ( -s , s)   =     \frac{1}{6} \lim_{s \to \frac16}  \frac{ A_4 ( -s, s)}{ \zeta( \frac32 - 3s ) \zeta ( \frac43 - 2s ) }    \\
= &  \frac{  \zeta(2)\zeta(\tfrac 53) }{6}   \prod_p  \Big(1-\frac1p \Big)     \Big(1-\frac 1{p^{\frac 23}}\Big)^2 \Big(1-\frac 1{p^{\frac 13}}\Big)      \Bigg\{   \frac{ \Big(1+\frac 1{p^{\frac 13}}\Big)^3  \Big(1 - \frac{1}{ p^{\frac23}}\Big)^2 }{  6  \Big(   1 - \frac{1}{ p^{\frac13}}\Big)^2 }     \\
& +  \frac{ \Big(1+ \frac1{p^{ \frac13}} \Big) \Big( 1+ \frac1{p^{ \frac23}} \Big)   \Big( 1 - \frac{1}{p^{\frac43}}\Big) }{2\Big( 1 - \frac{1}{p^{\frac23}}\Big) }  +  \frac{ \Big(1+ \frac1p  \Big)\Big(1 + \frac{1}{p^{\frac23}}+ \frac{1}{p^{\frac43}}\Big) }{3 \Big( 1 + \frac{1}{p^{\frac13}} + \frac{1}{p^{\frac23}} \Big)  }    + \frac{ \Big( 1 + \frac1{p^{ \frac13}} \Big)^2  \Big(  1- \frac{1}{p^{\frac23}} \Big)  }{p\Big(  1- \frac{1}{p^{\frac13}} \Big) } +  \frac{  1+ \frac1{p^{\frac13}}    }{p^2}  \Bigg\}  .
\end{align*}
The claim follows.
\end{proof}

We are now ready to estimate $J^\pm(X)$ when the support of $\widehat \phi$ is small.

\begin{lemma}
\label{lemma transition term small}
Let $\phi$ be a real even Schwartz function such that $\sigma=\sup({\rm supp} (\widehat \phi)) <1$. Let $J^\pm(X)$ be defined by \eqref{J X def}.
Then we have the estimate
$$J^\pm(X) =  C^\pm \phi \Big( \frac{L }{12 \pi i}\Big) X^{-\frac 13}+O_{\eps} \Big(X^{\frac {\sigma-1}2+\eps }\Big), $$ 
where 
\begin{equation} \label{equation definition C}
C^\pm  := \frac{5}{12} \frac{C_2^\pm }{C_1^\pm}   \frac{\Ga_{\pm}(\frac 13)}{\Ga_{\pm}(\frac 23)}   \frac{ \zeta(\tfrac 23)^2 \zeta(\tfrac 53) \zeta(2) }{  \zeta( \tfrac43)}   \prod_p \Big(1-\frac 1{p^{\frac 23}}\Big)^2   \Big(1- \frac1p \Big)\Big(1+\frac 2{p^{\frac 23}}+\frac 1p+\frac 1{p^{\frac 43}}\Big) . 
\end{equation}
\end{lemma}

\begin{proof}
We rewrite the integral in $J^\pm(X)$ as
\begin{multline}\label{J X integral eqn1}
\frac{1}{2 \pi i} \int_{(c)} (-2) \phi \Big(\frac{Ls}{2\pi i}\Big)\Big\{  \Big( 1-\frac{C_2^\pm}{C_1^\pm}X^{-\frac{1}{6}}\Big)\Big(-\frac{C_2^\pm}{C_1^\pm} X^{-\frac 16} \frac{\zeta'}{\zeta}(\tfrac 56+s)+  X^{-s}  \frac{\Ga_{\pm}(\frac 12-s)}{\Ga_{\pm}(\frac 12+s)}\zeta(1-2s)  \frac{A_3( -s, s)}{1-s} \Big)\\
+\frac{C_2^\pm }{C_1^\pm} X^{-s-\frac 16} \Big( 1-\frac{C_2^\pm}{C_1^\pm}X^{-\frac{1}{6}}\Big) \frac{\Ga_{\pm}(\frac 12-s)}{\Ga_{\pm}(\frac 12+s)}   \zeta(1-2s) \frac{\zeta(\tfrac 56-s)}{\zeta(\tfrac 56+s)} \frac{A_4(-s,s)}{ 1- \frac{6s}5} 
\\+\Big(\frac{C_2^\pm}{C_1^\pm}\Big)^2 X^{-s-\frac 13}   \frac{\Ga_{\pm}(\frac 12-s)}{\Ga_{\pm}(\frac 12+s)}\zeta(1-2s) \frac{ A_3( -s, s)}{1-s} 
\Big\}ds 
\end{multline}
for $ 0 < c < \frac16$. The integrand has a simple pole at $s=\frac 16$ with residue 
\begin{align} 
&   -2\phi \Big( \frac{L }{12 \pi i}\Big) \Big( 1-\frac{C_2^\pm}{C_1^\pm}X^{-\frac{1}{6}}\Big) X^{-\frac 16} \Big(\frac{C_2^\pm}{C_1^\pm}    - \frac25     \frac{\Ga_{\pm}(\frac 13)}{\Ga_{\pm}(\frac  23)}   \frac{ \zeta(\tfrac23 ) \zeta(3)}{   \zeta( \tfrac53 ) \zeta(2)}    \Big)  \notag \\
 & -2   \phi \Big( \frac{L }{12 \pi i}\Big) \frac{C_2^\pm }{C_1^\pm} X^{-\frac 13} \frac{\Ga_{\pm}(\frac 13)}{\Ga_{\pm}(\frac 23)}   \frac{5\zeta(\tfrac 23)^2}{4} \lim_{s\rightarrow \frac 16} (s-\tfrac 16)^2A_4(-s,s)\label{residue 16}
+O\Big(\phi\Big( \frac{L}{12 \pi i}\Big) X^{-\frac 12} \Big) \\
 &= -{C^\pm}  \phi \Big( \frac{L }{12 \pi i}\Big)  X^{- \frac13} + O( X^{\frac{\sigma}{6}- \frac12}) \notag
\end{align} 
by Lemma~\ref{lemma A4 limit at 16} as well as the fact that the first line vanishes. Due to Lemmas~\ref{lemma A_3 analytic continuation} and~\ref{lemma A_4 analytic continuation}, we can shift the contour of integration to the line $\Re(s)=\frac 12 - \frac{\eps}{2}$, at the cost of $-1$ times the residue~\eqref{residue 16}. 

We now estimate the shifted integral. The term involving $\frac{\zeta'}{\zeta}(\tfrac 56+s)$ can be evaluated by interchanging sum and integral; we obtain the identity
 \begin{equation}\label{equation log der zeta}
  \frac{1}{\pi i} \int_{(\frac 12-\frac \eps 2)} \phi \Big(\frac{Ls}{2\pi i}\Big)  \frac{\zeta'}{\zeta}(\tfrac 56+s)ds \\
= - \frac 2 L\sum_{p,e}\frac{\log p}{p^{\frac {5e}6}}\widehat\phi\Big(\frac{\log p^e}{L}\Big)  .
\end{equation}
The last step is to bound the remaining terms, which is carried out by combining~\eqref{equation decay phi complex} with Lemmas~\ref{lemma A_3 analytic continuation} and~\ref{lemma A_4 analytic continuation}.
\end{proof}
  
Finally, we complete the proof of Theorem \ref{theorem RC}.

\begin{proof}[Proof of Theorem \ref{theorem RC}]
Given Proposition~\ref{proposition RC} and Lemma~\ref{lemma transition term small}, the only thing remaining to prove is~\eqref{equation definition J(X)}. Applying~\eqref{J X integral eqn1} with $c=\frac 1{20}$ and splitting the integral into two parts, we obtain the identity
\begin{multline*} 
J^\pm (X)  =    \frac{2C_2^\pm X^{-\frac 16}}{C_1^\pm L}\Big( 1-\frac{C_2^\pm}{C_1^\pm}X^{-\frac{1}{6}}\Big) \sum_{p,e}\frac{\log p}{p^{\frac {5e}{6}}}\widehat\phi\left(\frac{\log p^e}{L}\right) \\
-\frac{1}{\pi i} \int_{(\frac{1}{20})} \phi \Big(\frac{Ls}{2\pi i}\Big)\Big\{  \Big( 1-\frac{C_2^\pm}{C_1^\pm}X^{-\frac{1}{6}}\Big)\Big(-\frac{C_2^\pm}{C_1^\pm} X^{-\frac 16} \frac{\zeta'}{\zeta}(\tfrac 56+s)+  X^{-s}  \frac{\Ga_{\pm}(\frac 12-s)}{\Ga_{\pm}(\frac 12+s)}\zeta(1-2s) \frac{ A_3( -s, s)}{1-s} \Big) \Big\} ds\\
-\frac{1}{\pi i} \int_{(\frac{1}{20})} \phi \Big(\frac{Ls}{2\pi i}\Big)\Big\{  \frac{C_2^\pm }{C_1^\pm} X^{-s-\frac 16} \Big( 1-\frac{C_2^\pm}{C_1^\pm}X^{-\frac{1}{6}}\Big) \frac{\Ga_{\pm}(\frac 12-s)}{\Ga_{\pm}(\frac 12+s)}   \zeta(1-2s) \frac{\zeta(\tfrac 56-s)}{\zeta(\tfrac 56+s)} \frac{A_4(-s,s)}{ 1- \frac{6s}5} 
\\+\Big(\frac{C_2^\pm}{C_1^\pm}\Big)^2X^{-s-\frac 13} \frac{\Ga_{\pm}(\frac 12-s)}{\Ga_{\pm}(\frac 12+s)}\zeta(1-2s) \frac{ A_3( -s, s) }{1-s} 
\Big\}ds.
\end{multline*}
By shifting the first integral to the line $\Re(s) = \frac15$ and applying \eqref{equation log der zeta}, we derive \eqref{equation definition J(X)}. Note that the residue at $s=\frac16$ is the first line of $(\ref{residue 16})$, which is equal to zero.
\end{proof}

\appendix

\section{Numerical investigations}



\label{appendix}
In this section we present several graphs\footnote{The computations associated to these graphs were done using development version 2.14 of pari/gp (see \url{https://pari.math.u-bordeaux.fr/Events/PARI2022/talks/sources.pdf}), and the full code can be found here: \url{https://github.com/DanielFiorilli/CubicFieldCounts} .} associated to the error term $$ E^+_{p} (X,T) := N^+_{p} (X,T) -A^+_p (T ) X -B^+_p(T) X^{\frac 56}.$$
We recall that we expect a bound of the form $ E^+_{p} (X,T) \ll_{\eps} p^{ \omega}X^{\theta+\eps} $ (see~\eqref{equation TT local 2}). Moreover, from the graphs shown in Figure~\ref{figure intro}, it seems likely that $\theta=\frac 12$ is admissible and best possible. 
Now, to test the uniformity in $p$, we consider the function
$$ f_p(X,T):=  \max_{1\leq x\leq X} x^{-\frac 12}|E^+_{p} (x,T)|;$$
we then expect a bound of the form $f_p(X,T) \ll_\eps p^{\omega}X^{\theta-\frac 12+\eps}$ with $\theta$ possibly equal to $\frac 12$. 
To predict the smallest admissible value of $\omega$, in Figure~\ref{figure T123}  we plot $f_{p}(10^4,T_j)$ for $j=1,2,3,$ as a function of $p<10^4$. 
\begin{figure}[h]

\begin{center}

\includegraphics[scale=.58]{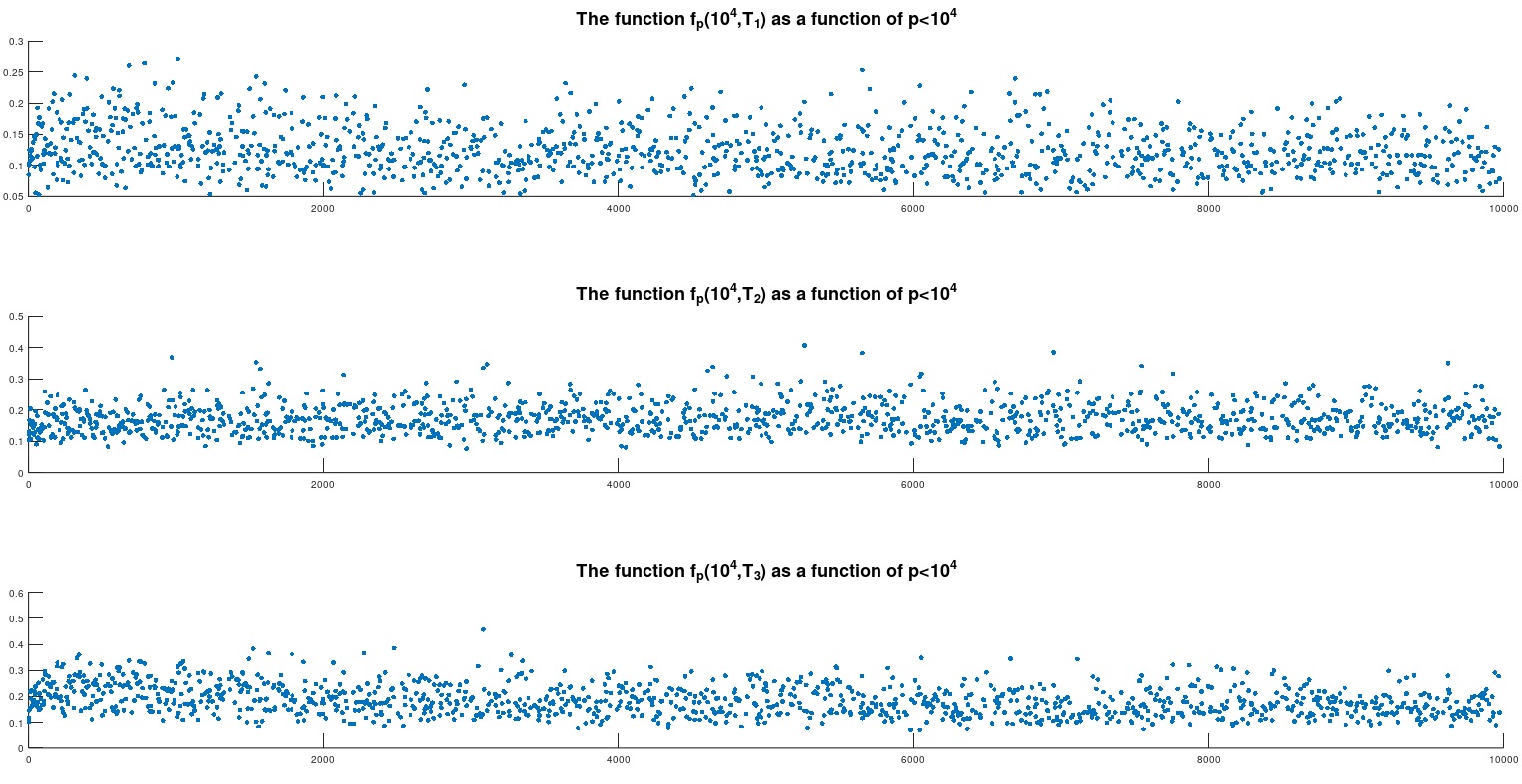}
\end{center}

\caption{A plot of $(p,f_{p}(10^4,T_j))$ for $p<10^4$ and $j=1,2,3.$ 
}\label{figure T123}
\end{figure}
From this data, it seems likely that any $\omega >0$ is admissible. Now, one might wonder whether this is still valid in the range $p>X$. To investigate this, in Figure~\ref{figure second T3} we plot the function $f_{p}(10^4,T_3)$ for every $10^4$-th prime up to $10^8$, revealing similar behaviour. Finally, we have also produced similar data associated to the quantity $N^-_{p}(X,T_j)$ with $j=1,2,3$, and the result was comparable to  Figure~\ref{figure T123}.

\begin{figure}[h]
\begin{center}
\includegraphics[scale=.55]{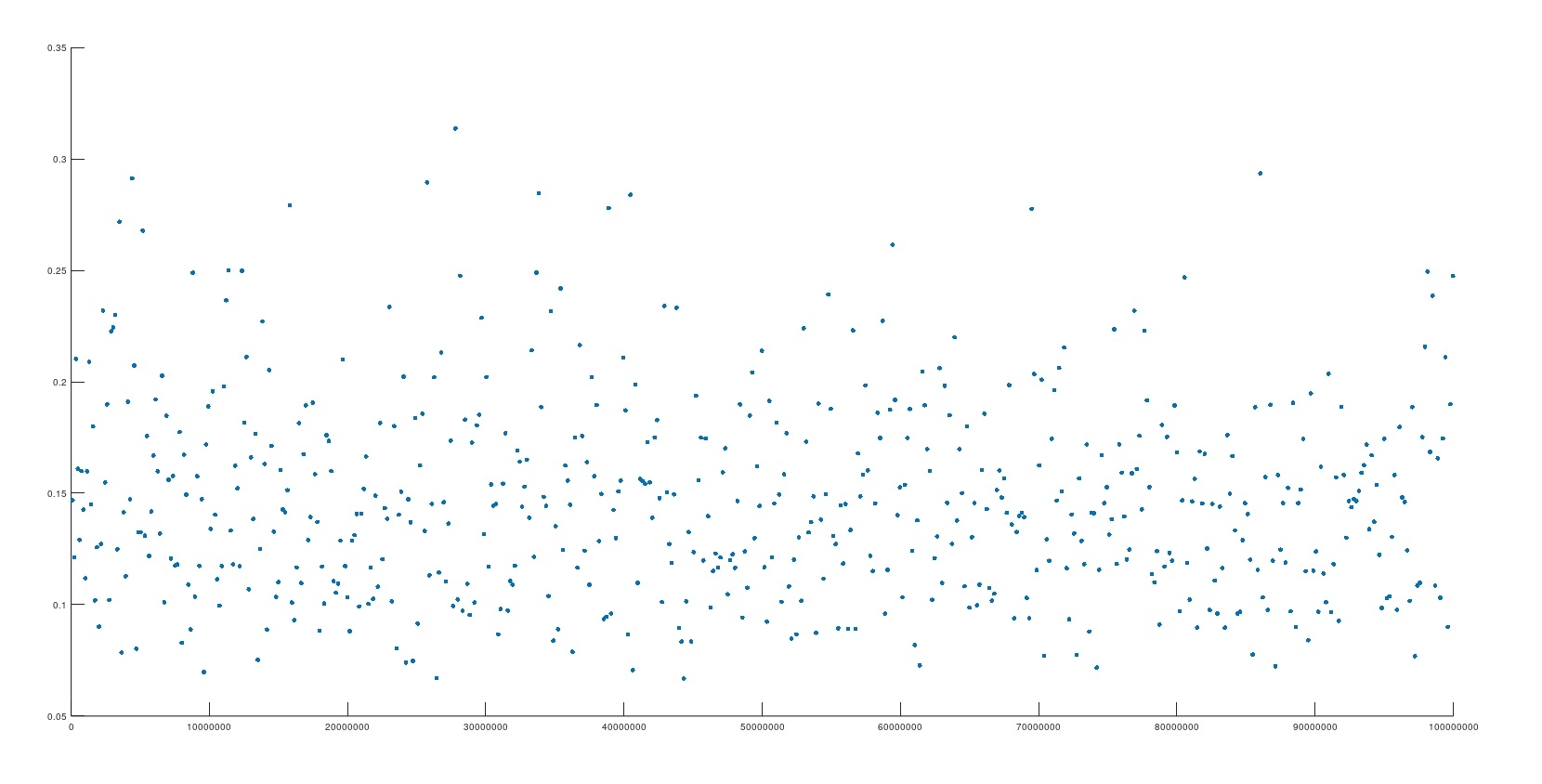}
\end{center}
\caption{A plot of some of the values of $(p,f_{p}(10^4,T_3))$ for $p<10^8$.
}
\label{figure second T3}
\end{figure}

However, it seems like the splitting type $T_4$ behaves  differently; see Figure~\ref{figure T45 first} for a plot of $p\cdot f_p(10^4,T_4)$ for every $p<10^5$.
\begin{figure}[h]
\begin{center}
\includegraphics[scale=.5]{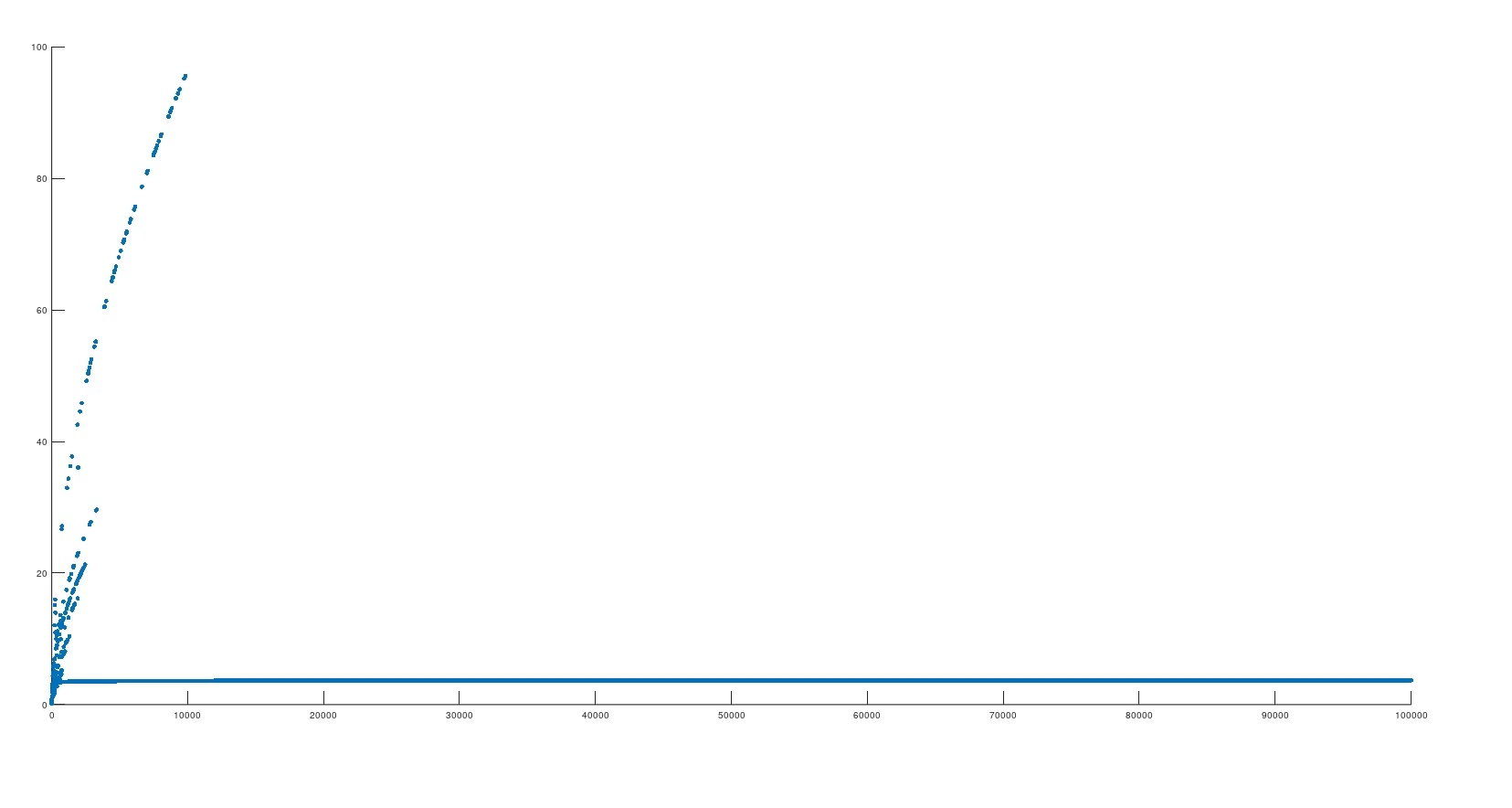}

\end{center}
\caption{A plot of $(p,pf_{p}(10^4,T_4))$ for $p<10^5$.
\label{figure T45 first}
}
\end{figure}
One can see that this graph is eventually essentially constant. This is readily explained by the fact that in the range $p>X$, we have $N^{\pm}_{p} (X,T_4) =0$. Indeed, if $p$ has splitting type $T_4$ in a cubic field $K$ of discriminant at most $X$, then $p$ must divide $D_K$, which implies that $p\leq X$. As a consequence, $pf_p(X,T_4)\asymp X^{\frac 12} $, which is constant as a function of $p$.  As for the more interesting range $p\leq X$, it seems like $f_p(X,T_4)\ll_\eps  p^{-\frac 12+\eps}X^\eps $ (i.e. for $T=T_4$, the values $\theta=\frac 12$ and any $\omega>-\frac 12$ are admissible in~\eqref{equation TT local 2}). In Figure~\ref{figure T4 second} we test this hypothesis with larger values of $X$ by plotting $p^{\frac 12} \cdot f_p(10^5,T_4)$ for all $p< 10^4$.
\begin{figure}[h]
\begin{center}
\includegraphics[scale=.55]{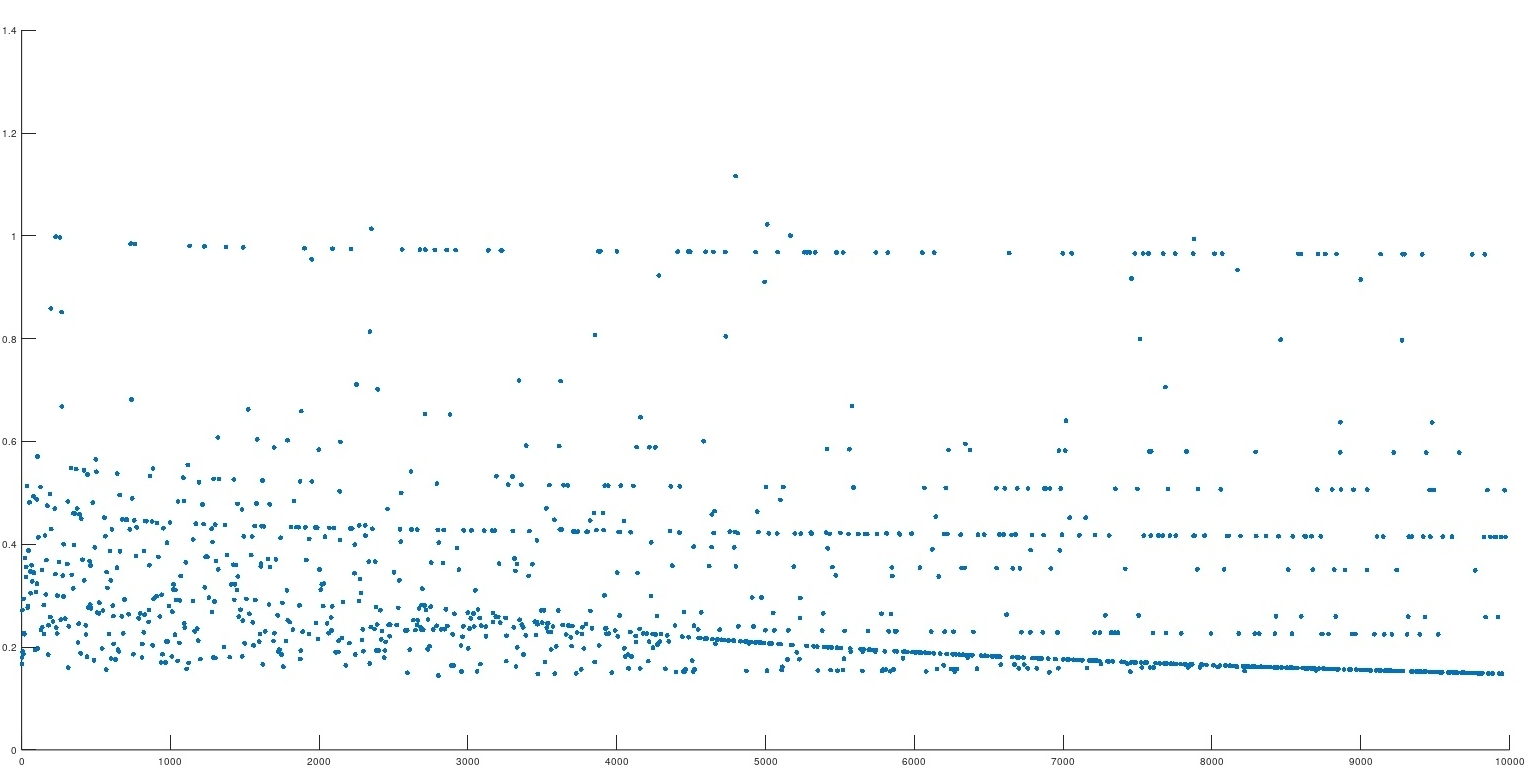}
\end{center}
\caption{A plot of $(p,p^{\frac 12}f_{p}(10^5,T_4))$ for $p<10^4$.
}
\label{figure T4 second}
\end{figure}
This seems to confirm that for $T=T_4$, the values $\theta=\frac 12$ and any $\omega>-\frac 12$ are admissible in~\eqref{equation TT local 2}. In other words, it seems like we have $E_p^+(X,T_4)\ll_\eps p^{-\frac 12+\eps} X^{\frac 12+\eps}$, and the sum of the two exponents here is $2\eps$, which is significantly smaller than the sum of exponents in Theorem~\ref{theorem omega result counts} which is  $\omega+\theta \geq \frac 12$. Note that this is not contradictory, since in that theorem we are assuming such a bound uniformly for all splitting types, and from the discussion above we expect that $E_p^+(X,T_1)\ll_\eps p^{\eps} X^{\frac 12+\eps}$ is essentially best possible.  Finally, we have also produced data for the quantity $N_p^-(X,T_4)$. The result was somewhat similar, but far from identical. We would require more data to make a guess as strong as the one we made for $E_p^+(X,T_4)$.

For the splitting type $T_5$, it seems like the error term is even smaller (probably owing to the fact that these fields are very rare). Indeed, this is what the graph of $p^2\cdot f_p(10^6,T_5)$ for all $p< 10^3$ in Figure~\ref{figure T5} indicates.
\begin{figure}[h]

\begin{center}
\includegraphics[scale=.5]{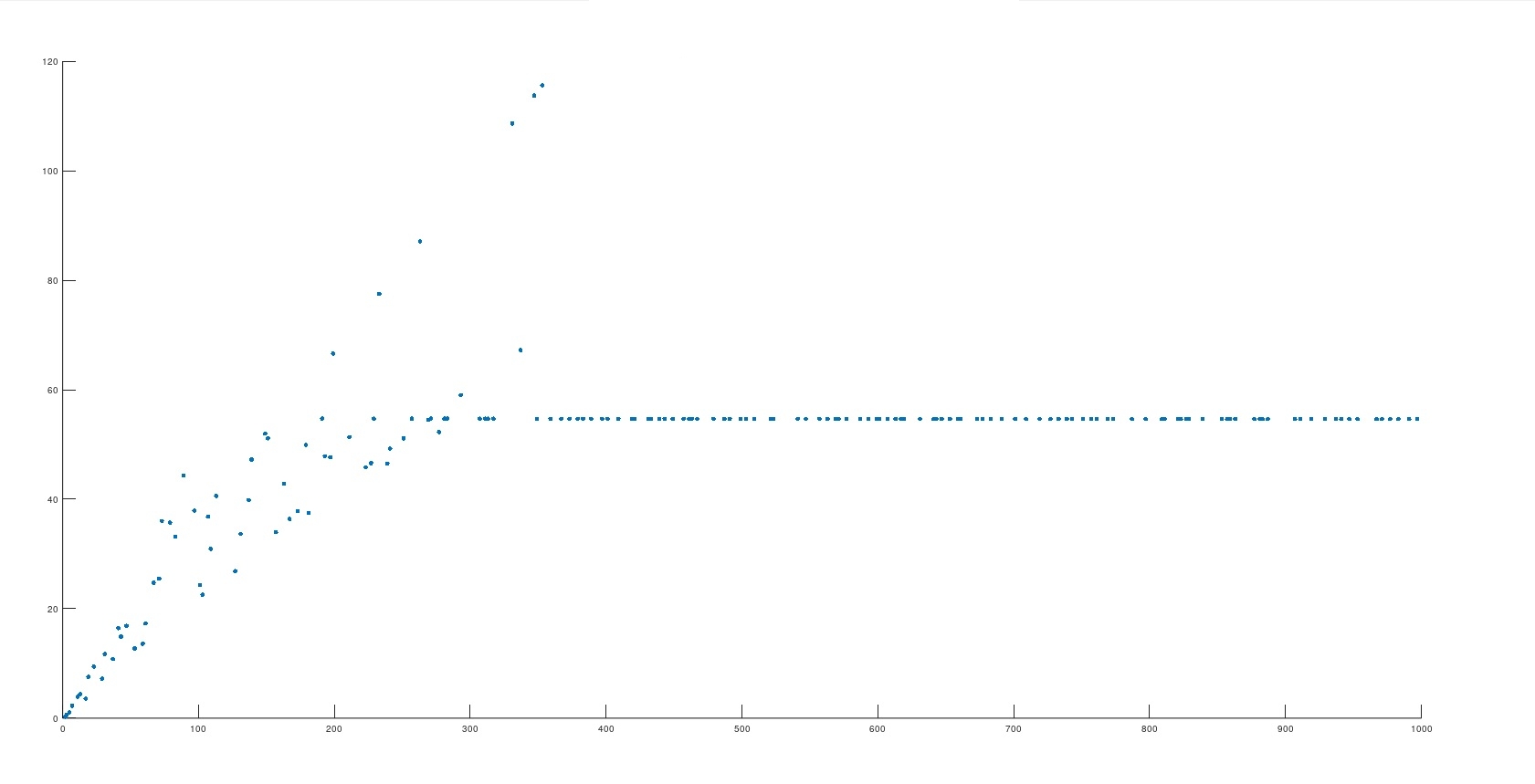}
\end{center}

\caption{A plot of $(p,p^{2}f_{p}(10^6,T_5))$ for $p<10^3$.
}
\label{figure T5}
\end{figure}
Again, there are two regimes. Firstly, by~\cite[p.\ 1216]{B}, $p>2$ has splitting type $T_5$ in the cubic field $K$ if and only if $p^2\mid D_K$, hence $N^{\pm}_p(X,T_5)=0$ for $p>X^{\frac 12}$ (that is $p^2\cdot f_p(X,T_5) \asymp X^{\frac 12}$). As for $p\leq X^{\frac 12}$, Figure~\ref{figure T5} indicates that $f_p(X,T_5) \ll_\eps p^{-1 +\eps}X^{\eps} $ (e.g. for $T=T_5$, the values $\theta=\frac 12$ and any $\omega>-1$ are admissible in~\eqref{equation TT local 2}). Once more, it is interesting to compare this with 
Theorem~\ref{theorem omega result counts}, since it seems like $ E^+_p(X,T_5) \ll_\eps p^{-1+\eps} X^{\frac 12+\eps} $, and the sum of the two exponents is now $-\frac 12+2\epsilon$.   We have also produced analogous data associated to the quantity $N_p^-(X,T_5)$. The result was somewhat similar.

Finally, we end this section with a graph (see Figure~\ref{figure E+(X)}) of $$E^+(X):=X^{-\frac 12} \big(N^+_{\rm all}(X)-C_1^+X-C_2^+X^{\frac 56}\big)$$ for $X< 10^{11}$ (which is the limit of Belabas' program\footnote{The program, based on the algorithm in~\cite{B}, can be found here: \url{https://www.math.u-bordeaux.fr/~kbelabas/research/cubic.html}} used for this computation). Here, $N_{\rm all}^+(X)$ counts all cubic fields of discriminant up to $X$, including Galois fields (by Cohn's work~\cite{C}, $N^+_{\rm all}(X) -N^+(X)\sim c X^{\frac 12}$, with $c=0.1585...$). This strongly supports the conjecture that $E^+(X)\ll_\eps X^{\frac 12+\eps}$ and that the exponent $\frac 12$ is best possible. It is also interesting that the graph is always positive, which is not without reminding us of Chebyshev's bias (see for instance the graphs in the survey paper~\cite{GM}) in the distribution of primes.

\begin{figure}[h] 
\begin{center}
\includegraphics[scale=.38]{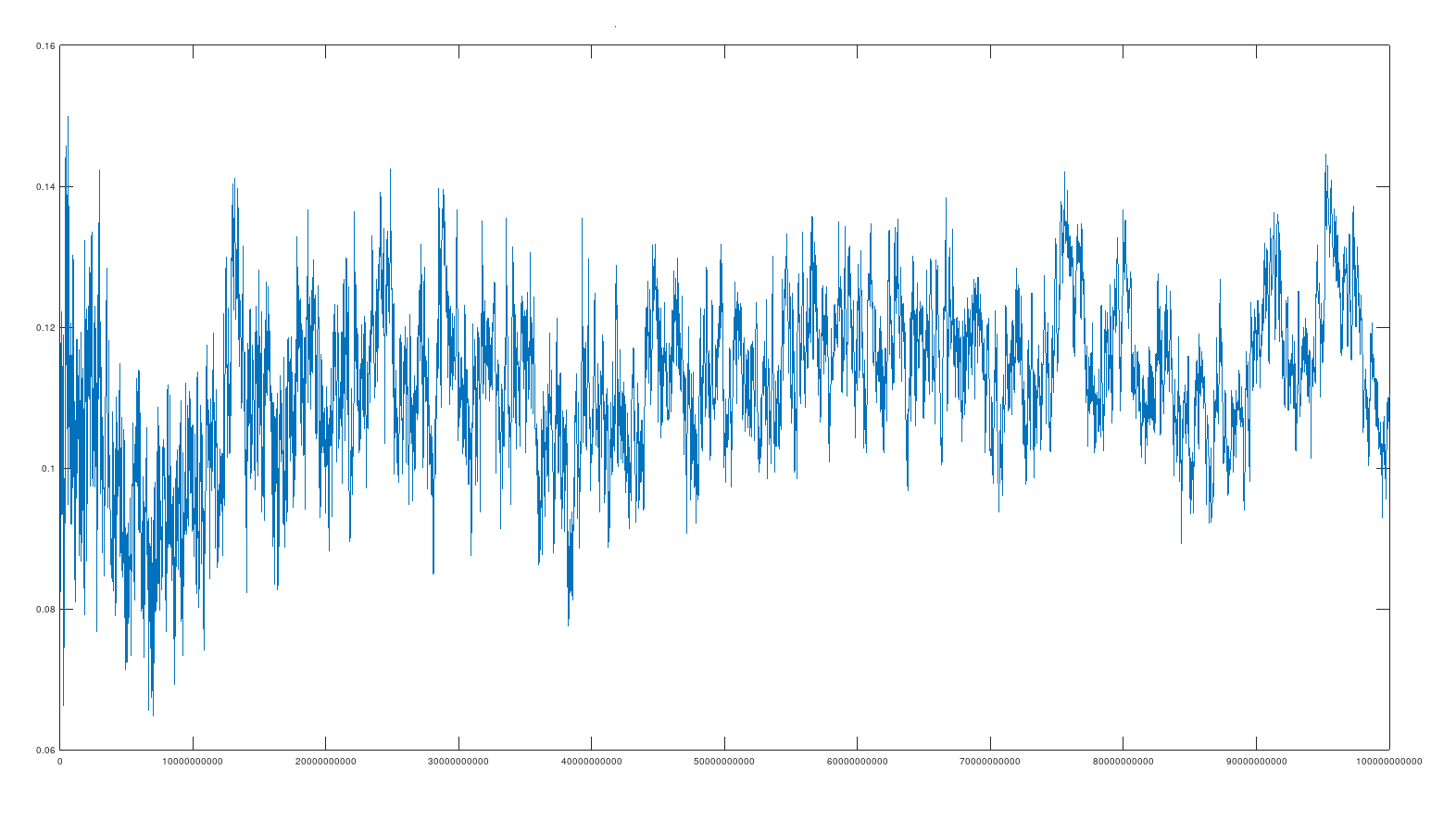}

\end{center}
\caption{A plot of $E^+(X) $ for $X<10^{11}$.
}
\label{figure E+(X)}
\end{figure}

Given this numerical evidence, one may summarize this section by stating that in all cases, it seems like we have square-root cancellation. More precisely, the data indicates that the bound
\begin{equation}
 N^+_{p} (X,T) -A^+_p (T ) X -B^+_p(T) X^{\frac 56} \ll_\eps (pX)^{\eps} \big(A^+_p (T ) X\big)^{\frac 12} 
\label{equation bound montgomery for cubic}
\end{equation}
could hold, at least for almost all $p$ and $X$. This is reminiscent of Montgomery's conjecture~\cite{Mo} for primes in arithmetic progressions, which states that
$$ \sum_{\substack{n\leq x \\ n\equiv a \bmod q}  }\Lambda(n)- \frac{x}{\phi(q)}  \ll_\eps x^{\eps} \Big(\frac{x}{\phi(q)}\Big)^{\frac 12} \qquad (q\leq x,\qquad (a,q)=1). $$
Precise bounds such as~\eqref{equation bound montgomery for cubic} seem to be far from reach with the current methods, however we hope to return to such questions in future work.

\end{document}